\documentclass[11pt]{article}%
\usepackage{amssymb,amsmath,amsfonts,amsthm,array,bm,bbm,color,stmaryrd}
\usepackage{amsmath}
\usepackage{amsfonts}
\usepackage{amssymb}
\usepackage{graphicx}%
\providecommand{\U}[1]{\protect\rule{.1in}{.1in}}
\numberwithin{equation}{section}
\newtheorem{theorem}{Theorem}[section]
\newtheorem{lemma}[theorem]{Lemma}
\newtheorem{corollary}[theorem]{Corollary}
\newtheorem{proposition}[theorem]{Proposition}
\newtheorem{remark}[theorem]{Remark}
\newtheorem{example}[theorem]{Example}
\newtheorem{definition}[theorem]{Definition}

\newtheorem{assumption}[theorem]{Assumption}
\def\<{\langle}
\def\>{\rangle}
\def\d{{\rm d}}
\def\L{\mathcal{L}}

\def\E{\mathbb{E}}

\def\P{\mathbb{P}}
\def\R{\mathbb{R}}
\def\T{\mathbb{T}}
\def\Z{\mathbb{Z}}

\def\eps{\varepsilon}

\begin{document}

\title{Quantitative convergence rates for scaling limit of SPDEs with transport noise}
\author{Franco Flandoli\thanks{Email: franco.flandoli@sns.it. Scuola Normale Superiore
of Pisa, Piazza dei Cavalieri 7, 56124 Pisa, Italy.} \quad Lucio
Galeati\thanks{Email: lucio.galeati@iam.uni-bonn.de. Institute for Applied
Mathematics, University of Bonn, Endenicher Allee 60, 53115 Bonn, Germany.}
\quad Dejun Luo\thanks{Email: luodj@amss.ac.cn. Key Laboratory of RCSDS,
Academy of Mathematics and Systems Science, Chinese Academy of Sciences,
Beijing 100190, China, and School of Mathematical Sciences, University of the
Chinese Academy of Sciences, Beijing 100049, China. } }
\maketitle

\begin{abstract}
We consider on the torus the scaling limit of stochastic 2D (inviscid) fluid dynamical
equations with transport noise to deterministic viscous equations.
Quantitative estimates on the convergence rates are provided by combining
analytic and probabilistic arguments, especially heat kernel properties and
maximal estimates for stochastic convolutions. Similar ideas are applied to
the stochastic 2D Keller-Segel model, yielding explicit choice of noise to
ensure that the blow-up probability is less than any given threshold.
Our approach also gives rise to some mixing property for stochastic linear transport equations
and dissipation enhancement in the viscous case.
\end{abstract}

\textbf{Keywords:} Transport noise, scaling limit, convergence rate, mixing, dissipation enhancement, stochastic convolution

\textbf{MSC (2020):} 60H15, 60H50

\section{Introduction}

This paper is mainly concerned with scaling limits for some nonlinear (inviscid) fluid equations perturbed by multiplicative noise of transport type. We already know that, under a suitable scaling of the noise, the stochastic equations converge weakly to some deterministic viscous equation, see e.g. \cite{Gal, FGL, Luo}. These results may be interpreted as the emergence of an \textit{eddy viscosity}, in a fluid with small scale turbulence. However, they are proved mainly by compactness method, and thus the rate of convergence is not known. Our first aim is to find explicit estimate on the convergence rate in such scaling limit results (see Section \ref{subsec-quantitat}). Furthermore, motivated by the above-mentioned works and also by \cite{FlaLuo}, we have shown in the recent paper \cite{FGL2} that transport noise enhances dissipation in a similar scaling limit, hence it can be used to suppress possible explosion of solutions to some nonlinear PDEs. Here we improve the result in \cite{FGL2} by providing quantitative estimates on the probability that the life time of solution is greater than some given $T$; this allows to choose the correct noise in order for the blow-up probability to be less than any initially chosen threshold (cf. Section \ref{subsec-estim-low-up}).

Next, we shall study the convergence rate of a linear inviscid transport equation, with transport noise, to a deterministic parabolic equation (see Section \ref{subsec-mixing} below). In this case, the property we prove is a quantitative version of the so-called mixing, where we can say precisely the \textit{rate of mixing} and the closedness, on a finite time horizon, to the decaying profile of the parabolic equation (cf. \cite{Dolgo, CKRZ, YaoZlatos, Alberti, BedroBlum, GessYar} for related results among the vast literature). Said differently, we prove that an \textit{eddy dissipation} emerges when the mixing is sufficiently turbulent. For linear viscous equation perturbed by the same transport noise, we show in Section \ref{subs-dissip-enhanc} the phenomenon of dissipation enhancement (see e.g. \cite{BedrCoti, FengIyer, CZDE, BedroBlum19}): the solutions converge to equilibrium in $L^2$-norm at an arbitrarily fast exponential speed for suitably chosen noise parameters.

Before giving more precise statements of the results obtained in this paper, let us introduce some frequently used notations.
Given positive numbers $a,b$, we write $a\lesssim b$ if there exists a
constant $C>0$ such that $a\leq C\,b$, $a\sim b$ if $a\lesssim b$ and
$b\lesssim a$; we write $a\lesssim_{\lambda}b$ to stress the dependence
$C=C(\lambda)$. We will mostly work on the torus $\mathbb{T}^{d}%
=\mathbb{R}^{d}/\mathbb{Z}^{d}$ and denote by $\{e_{k}\}_{k\in\mathbb{Z}^{d}}$
the standard Fourier basis $e_{k}(x)=e^{2\pi ik\cdot x}$; we denote by
$H^{s}(\mathbb{T}^{d};\mathbb{R}^{m})$ with $s\in\mathbb{R},m\geq1$ the usual
(possibly vector-valued) Sobolev spaces on $\mathbb{T}^{d}$. Whenever it does
not create confusion we will simply write $H^{s}$, similarly $L^{2}$ in place
of $L^{2}(\mathbb{T}^{d};\mathbb{R}^{m})$. We denote by $\<\cdot, \cdot\>$
the $L^{2}$-inner product or the duality between Sobolev spaces.

\subsection{Quantitative convergence rate for 2D fluid models}\label{subsec-quantitat}

In this section we consider on $\mathbb{T}^{2}$ the stochastic 2D Euler/Navier-Stokes equation
in vorticity form perturbed by transport noise:
\begin{equation}
\label{stoch-NS}%
\begin{cases}
\mathrm{d} \omega+ u\cdot\nabla\omega\,\mathrm{d} t + \circ\mathrm{d} W
\cdot\nabla\omega= \nu\Delta\omega\,\mathrm{d} t ,\\
u=K\ast\omega,
\end{cases}
\end{equation}
where $\nu\geq0$ ($\nu=0$ corresponding to stochastic Euler), $K$ is the Biot-Savart kernel on $\mathbb{T}^{2}$, and $u$ and $\omega$ are the velocity and vorticity of fluid; the stochastic differential is still understood in the Stratonovich sense. The noise $W$, parametrized by $\kappa>0$ and $\theta\in \ell^2(\Z^2_0)$, is defined as
  \begin{equation}\label{noise}
  W(t,x)=\sqrt{2\kappa}\sum_{k\in\mathbb{Z}_{0}^{2}}\theta_{k}\,\sigma_{k}(x)W_{t}^{k},
  \end{equation}
where $\{W^{k}\}_{k\in\mathbb{Z}_{0}^{2}}$ are standard complex Brownian
motions defined on some filtered probability space $(\Omega,\mathcal{F}%
,(\mathcal{F}_{t}),\mathbb{P})$, satisfying $\overline{W^{k}}=W^{-k}$, and
$\{\sigma_{k}\}_{k\in\mathbb{Z}_{0}^{2}}$ are given by
\[
\sigma_{k}(x)=a_{k}e_{k}(x),\quad k\in\mathbb{Z}_{0}^{2},
\]
where $a_{k}\in\mathbb{R}^{2}$ is a unit vector such that $a_{k}\cdot k=0$ and
$a_{k}=a_{-k}$. A typical choice is $a_k= \frac{k^\perp}{|k|}$ for $k\in \Z^2_+$
and $a_k= a_{-k}$ for $k\in \Z^2_-$, where $k^\perp= (k_2,-k_1)^\ast$ and
$\Z^2_0= \Z^2_+ \cup \Z^2_-$ is a partition of $\Z^2_0$ with $\Z^2_+ = - \Z^2_-$.
By construction, the family $\{\sigma_{k}\}_{k\in
\mathbb{Z}_{0}^{2}}$ is a CONS of the subspace of $L^{2}(\mathbb{T}%
^{2};\mathbb{C}^{2})$ consisting of mean zero, divergence free vector fields. It follows from
\eqref{noise} that $W$ is entirely characterized by the pair $(\kappa,\theta)$.

Throughout this paper we will always assume $\Vert\theta\Vert_{\ell^{2}}=1$,
which comes without loss of generality up to relabelling $\kappa$. We further
impose $\theta$ to be symmetric, i.e.
\begin{equation} \label{symmetry}
\theta_{k}=\theta_{l}\quad\mbox{for all }k,l\in\mathbb{Z}_{0}^{2}%
\mbox{ with }|k|=|l|.%
\end{equation}
Under this condition, it is easy to show that the first equation in \eqref{stoch-NS} has the It\^{o} form
\begin{equation}
\mathrm{d}\omega+u\cdot\nabla\omega\,\mathrm{d}t+\mathrm{d}W\cdot\nabla
\omega=(\kappa+\nu)\Delta\omega\,\mathrm{d}t.\label{stoch-NS-Ito}%
\end{equation}
Given $\omega_{0}\in L^{2}$, this equation admits a weak solution (strong in
the probabilistic sense if $\nu>0$) satisfying
\begin{equation}
\sup_{t\geq0}\bigg\{\Vert\omega_{t}\Vert_{L^{2}}^{2}+2\nu\int_{0}^{t}%
\Vert\nabla\omega_{t}\Vert_{L^{2}}^{2}\,\mathrm{d}t\bigg\}\leq\Vert\omega
_{0}\Vert_{L^{2}}^{2}\quad\mathbb{P}\mbox{-a.s.}\label{solu-bounds}%
\end{equation}
According to \cite{FGL}, if we take a sequence $\{\theta^{n}\}_{n}\subset
\ell^{2}$ such that
\begin{equation}
\Vert\theta^{n}\Vert_{\ell^{2}}=1\ (\forall\,n\geq1),\quad\lim_{n\rightarrow
\infty}\Vert\theta^{n}\Vert_{\ell^{\infty}}=0,\label{theta-n}%
\end{equation}
and consider the solutions $\omega^{n}$ of \eqref{stoch-NS} corresponding to
$\theta^{n}$, then for any $\alpha>0$, $\omega^{n}$ converges in probability,
in the topology of $C([0,T];H^{-\alpha})$, to the unique solution of the
deterministic 2D Navier-Stokes equation
\begin{equation}
\partial_{t}\tilde{\omega}+\tilde{u}\cdot\nabla\tilde{\omega}=(\kappa
+\nu)\Delta\tilde{\omega},\quad\tilde{\omega}|_{t=0}=\omega_{0}%
,\label{determ-NS}%
\end{equation}
where $\tilde{u}=K\ast\tilde{\omega}$.

Our first main result gives an explicit estimate on the distance between the
solutions of \eqref{stoch-NS} and \eqref{determ-NS}.

\begin{theorem}
\label{thm-1} Let $\omega$ and $\tilde\omega$ be weak solutions to
\eqref{stoch-NS} and \eqref{determ-NS} respectively and assume $\omega$
satisfies \eqref{solu-bounds}. Then, for any $\alpha\in(0,1)$, there exists
some $C=C(\alpha)>0$ such that for any $\varepsilon\in(0,\alpha]$, one has

\begin{itemize}
\item[\textrm{(i)}] for any $\nu\geq0$,
\[
\mathbb{E}\Big[\|\omega-\tilde\omega\|_{C([0,T];H^{-\alpha})}^{p} \Big]^{1/p}
\lesssim_{\varepsilon, p,T} \kappa^{\varepsilon/2} \|\theta\|_{\ell^{\infty}%
}^{\alpha-\varepsilon} \|\omega_{0}\|_{L^{2}} \exp\bigg[C \frac{1+
T(\kappa+\nu)}{(\kappa+\nu)^{2}} \|\omega_{0}\|_{L^{2}}^{2} \bigg];
\]

\item[\textrm{(ii)}] if $\nu>0$, then
\[
\mathbb{E}\Big[\|\omega-\tilde\omega\|_{C([0,T];H^{-\alpha})}^{p} \Big]^{1/p}
\lesssim_{\varepsilon, p,T} \kappa^{\varepsilon/2} \|\theta\|_{\ell^{\infty}%
}^{\alpha-\varepsilon} \|\omega_{0}\|_{L^{2}} \exp\bigg[ \frac{C}{\nu}
\|\omega_{0}\|_{L^{2}}^{2} \bigg].
\]

\end{itemize}
\end{theorem}

In the above estimates, the implicit constants behind the notation
$\lesssim_{\varepsilon, p,T}$ might explode as $\varepsilon\to0$ or
$p,T\to\infty$. Below we give some more comments on the results.

\begin{remark}
\begin{itemize}

\item[\textrm{(1)}] The first estimate is stable in the vanishing viscosity
limit, indeed there is no harm in taking $\nu=0$ as long as $\kappa>0$; the
second one instead has the advantage that $T$ does not appear in the exponential.

\item[\textrm{(2)}] The revelant information in the above estimates comes from
values of $\alpha$ as large as possible, $\alpha\sim1-\varepsilon$. Indeed the
statement for values $\alpha^{\prime}<\alpha$ follows from interpolating the
estimate for $\alpha$ with the uniform bound
\[
\| \omega-\tilde{\omega}\|_{C([0,T];L^{2})} \leq2 \|\omega_{0}\|_{L^{2}}%
\quad\mathbb{P}\mbox{-a.s.}
\]
Similarly, in the case $\nu>0$ one can obtain rates of convergence in
$L^{2}(0,T;H^{s})$ for any $s\in(0,1)$, by interpolating with the uniform
bound in $L^{2}(0,T;H^{1})$ coming from \eqref{solu-bounds}.

\item[\textrm{(3)}] In the above result we assumed for simplicity $\omega
_{0}=\tilde{\omega}_{0}$. If this is not the case, the second estimate
becomes
\[
\mathbb{E}\Big[ \|\omega-\tilde\omega\|_{C([0,T];H^{-\alpha})}^{p} \Big]^{1/p}
\lesssim_{\varepsilon,p,T} \bigg[ \|\omega_{0} - \tilde\omega_{0}%
\|_{H^{-\alpha}}+ \kappa^{\varepsilon/2} \|\theta\|_{\ell^{\infty}}%
^{\alpha-\varepsilon} \|\omega_{0}\|_{L^{2}} \bigg] \exp\bigg[\frac{C}{\nu}
\|\omega_{0}\|_{L^{2}}^{2} \bigg];
\]
similar changes apply to the first estimate.
\end{itemize}
\end{remark}

In practice, we take a sequence of $\{\theta^{n} \}_{n} \subset\ell^{2}$ with
the properties \eqref{symmetry} and \eqref{theta-n}, and consider the
corresponding stochastic equations \eqref{stoch-NS} with $\theta^{n}$ in place
of $\theta$. This together with Theorem \ref{thm-1} will give us explicit
rates of convergence, as shown by the next example.

\begin{example}\label{ex:coefficients}
Consider $\theta^{n}_{k} := \tilde\theta^{n}_{k} /
\|\tilde\theta^{n}\|_{\ell^{2}}$ for some $\tilde\theta^{n} \in\ell^{2},\,
n\geq1$. Let $a\geq0$.

\begin{itemize}
\item[\textrm{(1)}] Define
\[
\tilde\theta^{n}_{k}= \frac1{|k|^{a}} \mathbf{1}_{\{n\leq|k|\leq2n\}}, \quad
k\in\mathbb{Z}^{2}_{0}.
\]
Then $\|\tilde\theta^{n}\|_{\ell^{\infty}}= n^{-a}$ and $\|\tilde\theta
^{n}\|_{\ell^{2}} \sim n^{1-a}$, thus
\[
\|\theta^{n}\|_{\ell^{\infty}} = \frac{\|\tilde\theta^{n}\|_{\ell^{\infty}}%
}{\|\tilde\theta^{n}\|_{\ell^{2}}} \sim\frac1n.
\]
We see that the convergence rate is the same for different value of $a\geq0$.

\item[\textrm{(2)}] Define
\[
\tilde\theta^{n}_{k}= \frac1{|k|^{a}} \mathbf{1}_{\{1\leq|k|\leq n\}}, \quad
k\in\mathbb{Z}^{2}_{0},
\]
then $\|\tilde\theta^{n}\|_{\ell^{\infty}}=1$ and
\[
\|\tilde\theta^{n}\|_{\ell^{2}} \sim%
\begin{cases}
n^{1-a}, & 0\leq a<1;\\
\sqrt{\log n}, & a=1;\\
1, & a>1.
\end{cases}
\]
Hence,
\[
\|\theta^{n}\|_{\ell^{\infty}}= \frac{\|\tilde\theta^{n}\|_{\ell^{\infty}}%
}{\|\tilde\theta^{n}\|_{\ell^{2}}} \sim%
\begin{cases}
n^{a-1}, & 0\leq a<1;\\
(\log n)^{-1/2}, & a=1;\\
1, & a>1.
\end{cases}
\]
The convergence rate in this case strongly depends on $a$. Note that in the
case $a>1$ there is no rate of convergence, indeed we are not in the
hypothesis for the scaling limit to hold; we list it here for the sake of completeness.
\end{itemize}
\end{example}

We provide a simple consequence of the above result to show the power of the
quantitative convergence rates. In the rest of this subsection, we assume
$\nu=0$ and thus we are considering stochastic 2D Euler equations in
\eqref{stoch-NS}. In this case, it is well known that the uniqueness of
solutions remains open for the deterministic 2D Euler equation with $L^{2}%
$-initial vorticity. We have discussed in \cite[Section 6.1]{FGL} the
``approximate uniqueness'' of weak solutions to stochastic 2D Euler equations.
Roughly speaking, it means that the distances between weak solutions of
\eqref{stoch-NS} will vanish if we take a sequence of $\theta^{n}$ as in the
example above. This follows from the scaling limit result since the weak
solutions of \eqref{stoch-NS} converge weakly to the unique solution of the
deterministic 2D Navier-Stokes equation \eqref{determ-NS} with $\nu=0$; see
\cite[Corollary 6.3]{FGL} for a qualitative statement. Thanks to Theorem
\ref{thm-1}, we can now provide a more explicit estimate on the distances
between weak solutions.

To this end, we fix some $\alpha\in(0,1)$ and denote by $\mathcal{X}:=
C([0,T],H^{-\alpha})$; we also write $\mathcal{L}_{\theta}$ for the collection
of laws of weak solutions to \eqref{stoch-NS}, with a fixed initial condition
$\omega_{0} \in L^{2}(\mathbb{T}^{2})$; $\mathcal{L}_{\theta}$ can be regarded
as a subset of the space $\mathcal{P}(\mathcal{X})$ of Borel probability
measures on $\mathcal{X}$. We endow the space $\mathcal{P}(\mathcal{X})$ with
the Wasserstein distance: for $Q,Q^{\prime}\in\mathcal{P}(\mathcal{X})$,
\[
d_{p}(Q,Q^{\prime})= \bigg[ \inf_{\pi\in\mathcal{C}(Q,Q^{\prime})}
\int_{\mathcal{X }\times\mathcal{X}} \|\omega- \omega^{\prime}\|_{\mathcal{X}%
}^{p} \,\mathrm{d}\pi(\omega, \omega^{\prime}) \bigg]^{1/p} ,
\]
where $\mathcal{C}(Q,Q^{\prime})$ is the collection of probability measures on
$\mathcal{X }\times\mathcal{X}$ having $Q$ and $Q^{\prime}$ as the first and
the second marginal measures. Then we have the following simple result which
quantifies the distances between elements in $\mathcal{L}_{\theta}$.

\begin{corollary}
\label{cor-approx-uniq} Assume $\nu=0$ in \eqref{stoch-NS}; then for any
$p\geq1$ and $T>0$, we have
\[
d_{p}(Q,Q^{\prime}) \lesssim_{p,T} \kappa^{\alpha/4} \|\theta\|_{\ell^{\infty
}}^{\alpha/2} \|\omega_{0}\|_{L^{2}} \exp\bigg[\frac{C}{\kappa^{2}}
(1+T\kappa) \|\omega_{0}\|_{L^{2}}^{2} \bigg]
\]
for any $Q,Q^{\prime}\in\mathcal{L}_{\theta}$. In particular, assuming that
$\theta^{n}$ is defined as in (1) of Example \ref{ex:coefficients}, then
\[
d_{p}(Q_{n},Q^{\prime}_{n}) \lesssim_{p,T} \frac{\kappa^{\alpha/4}}%
{n^{\alpha/2}} \|\omega_{0}\|_{L^{2}} \exp\bigg[\frac{C}{\kappa^{2}}
(1+T\kappa) \|\omega_{0}\|_{L^{2}}^{2} \bigg]
\]
for any $Q_{n},Q^{\prime}_{n}\in\mathcal{L}_{\theta_{n}}$.
\end{corollary}

The results in this part will be proved in Section \ref{subsec-proof}. We
remark that our method for deriving estimates in Theorem \ref{thm-1} is quite
general (see Section \ref{subsec-heuristic} for a brief description), thus we
can deal with other fluid models such as the 2D Boussinesq system and the
modified Surface Quasi-Geostrophic (mSQG for short) equations. We will present
the related results in Sections \ref{sec:boussinesq} and \ref{sec:mSQG}.

\subsection{Explicit estimates on blow-up probability}

\label{subsec-estim-low-up}

In this part we are concerned with nonlinear PDEs exhibiting a dichotomy
between global solutions for small initial data and finite time blow-up for
large ones. A famous example is the 2D Keller-Segel system (cf. \cite{Patlak,
KS70, KS71})
\begin{equation}
\label{keller-segel}%
\begin{cases}
\partial_{t} \rho= \Delta\rho-\chi\nabla\cdot(\rho\nabla c),\\
-\Delta c = \rho-\rho_{\Omega},
\end{cases}
\end{equation}
where $\Omega\subset\mathbb{R}^{2}$ is a regular bounded domain and
$\rho_{\Omega} =\int_{\Omega} \rho(x)\, \mathrm{d} x$. Here $\rho:\Omega
\to\mathbb{R}$ describes the evolution of a bacterial population density whose
motion is biased by the density of a chemoattractant $c:\Omega\to\mathbb{R}$
produced by the population itself; $\chi>0$ is a fixed sensitivity parameter,
which will be taken as 1 in the sequel. It was shown in \cite{JL92} that for
$\rho_{\Omega}(0)$ below a critical threshold, global existence of regular
solutions $(\rho,c)$ holds, while there exist radially symmetric solutions
blowing up in finite time if $\Omega$ is a disk; see also \cite[Theorem
8.1]{KX16} for similar examples when $\Omega=\mathbb{T}^{2}$. The blow-up
mechanism is due to mass concentration and formation of Diracs for $\rho$.

In the recent paper \cite{FGL2}, we have shown that transport noise delays
blow-up for large initial data with high probability; this idea works for a
large class of nonlinear PDEs, including \eqref{keller-segel}. However, as in
\cite{FGL, FlaLuo}, the method in \cite{FGL2} is again based on compactness
arguments and thus we were only able to prove that, under some natural
conditions on the nonlinear term, for any given $T>0$ and small $\varepsilon
>0$, there exist $\kappa>0$ and $\theta\in\ell^{2}$ (determining the noise $W$
in \eqref{eq:noise-higherdim}) such that the corresponding solution has a life time greater
than $T$, with probability greater than $1-\varepsilon$. In this work, we will
give quantitative estimate on the blow-up probability, which in turn yields
explicit choices of $(\kappa,\theta)$.

As in \cite{FGL2}, the approach works for a large class of PDEs, but we will
mostly focus on equation \eqref{keller-segel} in order to convey the main
ideas underlying it. We fix the domain $\Omega=\mathbb{T}^{2}$ with periodic
boundary condition, $\chi=1$ and consider the stochastic Keller-Segel system
\begin{equation}
\label{stoch-keller-segel}%
\begin{cases}
\mathrm{d} \rho= \big[\Delta\rho- \nabla\cdot(\rho\nabla c)\big]\, \mathrm{d}
t + \circ\mathrm{d} W\cdot\nabla\rho\\
-\Delta c = \rho- \rho_{\mathbb{T}^{2}}%
\end{cases}
\end{equation}
with initial data $\rho_{0}\in L^{2}$; the noise is given as in \eqref{noise}.

Since $W$ is spatially divergence free, the stochastic equation enjoys the
same energy estimate as \eqref{keller-segel} and hence local existence and
uniqueness of a maximal solution can be shown similarly to the deterministic
case (see \cite{FGL2} for a rigorous proof); still, solutions might blow up in
finite time. Given initial value $\rho_{0}\in L^{2}$, $\theta\in\ell^{2}$ and
noise intensity $\kappa>0$, we denote by $\tau(\rho_{0};\theta, \kappa)$ the
blow-up time of the unique local solution to \eqref{stoch-keller-segel}. Here is the last main result of the paper.

\begin{theorem}
\label{thm-3} Fix $\varepsilon\in(0,1)$, $p\in[1,\infty)$, $L,T>0$. Then there
exist constants $C_{1}$ and $C_{2}=C_{2}(\varepsilon,p,L,T)$ with the
following property: for any tuple $(\rho_{0},\kappa,\theta)$ such that $\|
\rho_{0}\|_{L^{2}}\leq L$, $\kappa\geq1 + C_{1} L^{2}$ and $\theta\in\ell^{2}$
satisfying \eqref{symmetry} with $\| \theta\|_{\ell^{2}}=1$, it holds
\[
\mathbb{P}(\tau(\rho_{0}; \theta,\kappa)< T) \leq C_{2}\, \kappa^{\varepsilon
p/4}\, \| \theta\|_{\ell^{\infty}}^{p(1-\varepsilon)}.
\]

\end{theorem}

\begin{remark}
The constant $C_{1}$ does not depend on $(\varepsilon,p,L,T)$, but it depends
on $d=2$ and the domain $\Omega=\mathbb{T}^{2}$, due to the application of
Poincar\'e inequality and Sobolev embedding in the proof. Both constants
$C_{1}$ and $C_{2}$ can be calculated explicitly, see Section
\ref{sec:blow-up} for more details. It follows from the above estimate that,
once $\kappa$ is fixed as above, choosing $\theta^{n}$ as in Example
\ref{ex:coefficients}-(1) yields
\[
\mathbb{P}(\tau(\rho_{0};\theta^{n},\kappa)<T) \lesssim n^{-p(1-\varepsilon)}
\]
i.e. the probability of blow-up before time $T$ decreases with arbitrarily
high polynomial rate.
\end{remark}

\subsection{Some results on mixing and dissipation enhancement}

It turns out that, in the linear case, our approach also leads to some interesting (though possibly weaker) results on the mixing property and dissipation enhancement due to transport noise. In this section, we shall work on the general torus $\mathbb{T}^{d}\, (d \geq 2)$ and consider the following noise which is similar to \eqref{noise}:
  \begin{equation}\label{eq:noise-higherdim}
  W(t,x)= \sqrt{C_{d} \kappa}\, \sum_{k\in
  \mathbb{Z}^{d}_{0}}\sum_{i=1}^{d-1} \theta_{k} \sigma_{k,i}(x) W^{k,i}_{t},
  \end{equation}
where $C_{d}=d/(d-1)$ is a normalizing constant, $\kappa>0$ is still the noise intensity and $\theta\in\ell^{2} =\ell^{2}(\mathbb{Z}_{0}^{d})$ is symmetric in $k\in \Z_0^d$; $\{W^{k,i}:k\in\mathbb{Z}^{d}_{0}, i=1,\ldots,d-1\}$ are standard complex Brownian motions satisfying
\[
\overline{W^{k,i}} = W^{-k,i}, \quad\big[W^{k,i},W^{l,j} \big]_{t}= 2t
\delta_{k,-l} \delta_{i,j};
\]
$\{\sigma_{k,i}: k\in\mathbb{Z}^{d}_{0}, i=1,\ldots,d-1\}$ are divergence free vector fields on $\T^d$ defined as
  $$\sigma_{k,i}(x) = a_{k,i} e_{k}(x),$$
where $\{a_{k,i}\}_{k,i}$ is a subset of the unit sphere $\mathbb{S}^{d-1}$ such that: i) $a_{k,i}=a_{-k,i}$ for all
$k\in\mathbb{Z}^{d}_{0}$; ii) for fixed $k$, $\{a_{k,i}\}_{i=1}^{d-1}$ is a
ONB of $k^{\perp}=\{y\in\mathbb{R}^{d}:y\cdot k=0 \}$. In this way,
$\{\sigma_{k,i}\}_{k,i}$ is a CONS of the subspace of $L^{2}(\mathbb{T}%
^{d};\mathbb{C}^{d})$ of mean zero, divergence free vector fields.

\subsubsection{Quantitative finite horizon mixing}\label{subsec-mixing}

In the first part, we consider on $\mathbb{T}^{d} $ the stochastic transport equation
\[
\mathrm{d}f+\circ\mathrm{d}W\cdot\nabla f=0
\]
which has the It\^{o} form
\begin{equation}\label{stoch-transp-Ito}
\mathrm{d}f+\mathrm{d}W\cdot\nabla f=\kappa\Delta f\,\mathrm{d}t.
\end{equation}
Given $f_{0}\in L^{\infty}$, this equation admits a weak $L^{\infty}$-solution
satisfying
  \begin{equation} \label{solu-bounds transp}
  \mathbb{P}\mbox{-a.s.}\quad \sup_{t\geq0} \Vert f_{t}\Vert_{L^{p}}=\Vert f_{0}\Vert_{L^{p}}%
  \end{equation}
for every $p\in\left[  1,\infty\right]$.
Indeed, if $\theta$ enjoys suitable summability (e.g. $\sum_k |k|^2 \theta_k^2 <\infty$) we can construct the stochastic flow $\{X_t\}_{t\geq 0}$ associated to $W$ and represent the solution as $f_t(x) = f_0(X^{-1}_t(x))$ (see for instance \cite[Proposition 2.3]{Zha}); $W$ being divergence-free implies incompressibility of $X_t$ and thus \eqref{solu-bounds transp}. The result can then be generalized to any $\theta\in\ell^2$ by classical compactness arguments.

The expected value $\overline
{f}_{t}=\mathbb{E}\left[  f_{t}\right]  $ is a weak $L^{\infty}$-solution of
the deterministic heat equation%
\begin{equation}\label{heat-eq}
\partial_{t}\overline{f}_{t}=\kappa\Delta\overline{f}_{t}%
\end{equation}
which decays exponentially in $L^{2}$-norm, as opposed to $f_{t}$ which is
$L^{2}$-norm-preserving. However, in the weak sense, $f_{t}$ and $\overline
{f}_{t}$ are close to each other if $\Vert\theta\Vert_{\ell^{\infty}}$ is small;
more precisely:

\begin{theorem}\label{thm-transport}
For every $\phi\in L^{2}(\mathbb{T}^{d})$ and all $t\geq0$,
\begin{equation}\label{eq:thm-transport-eq1}
\mathbb{E}\left[  \left\vert \left\langle f_{t},\phi\right\rangle
-\left\langle \overline{f}_{t},\phi\right\rangle \right\vert ^{2}\right]
\leq\Vert\theta\Vert_{\ell^{\infty}}^{2}\Vert f_{0}\Vert_{L^{\infty}}^{2}%
\Vert\phi\Vert_{L^{2}}^{2}.
\end{equation}
If $\chi$ is a smooth mollifier, then%
\begin{equation}\label{smeared}
\mathbb{E}\left[  \Vert \chi\ast f_{t}- \chi\ast\overline{f}_{t}\Vert_{L^{2}}%
^{2}\right]  \leq\Vert\theta\Vert_{\ell^{\infty}}^{2}\Vert f_{0}%
\Vert_{L^{\infty}}^{2}\Vert \chi\Vert_{L^{2}}^{2}.%
\end{equation}

\end{theorem}

The proof will be presented in Section \ref{subs-inviscid-transport}; see \cite[Theorem 1.1]{FGL3} for related results concerning stochastic heat equations on bounded domains, where the difficulty lies in dealing with Dirichlet boundary condition.

The above result
could be interpreted under the light of the concept of mixing of
passive scalars under suitable transport coefficients. Strictly speaking, what
is usually called mixing is the property
\[
\lim_{t\rightarrow\infty}\left\langle f_{t},\phi\right\rangle =0
\]
for all test functions $\phi$ of a suitable class, reformulated also by means
of decay to zero of negative Sobolev norms and improved to exponential decays
in most of the available results. The deterministic literature on the subject
is very large, see for instance \cite{Alberti, BedrCoti}. Mixing by
random transport has been proved in two outstanding works, first in the case
of white noise in time (as in our model) in \cite{Dolgo}, then in the case
when the random velocity field is the solution of stochastic equations,
including Navier-Stokes, see \cite{BedroBlum}. Compared to such mixing
results, the above theorem misses the decay at infinity; our result only claims that
$\left\langle f_{t},\phi\right\rangle $ decays similarly to $\left\langle
\overline{f}_{t},\phi\right\rangle $ as soon as $\Vert\theta\Vert
_{\ell^{\infty}}$ is so small that the number $\Vert\theta\Vert
_{\ell^{\infty}}^{2}\Vert f_{0}\Vert_{L^{\infty}}^{2}\Vert\phi\Vert_{L^{2}%
}^{2}$ is smaller than $e^{-\kappa t}$ (hence only on a finite time horizon).
However, it contributes additional information and it has some advantages: i)
the decay rate $\kappa$ is related to the noise in a very simple way; ii) the
shape of the random process $f_{t}$, suitably smeared (see \eqref{smeared}), is
close to the shape of the decaying solution $\overline{f}_{t}$ of the heat
equation; iii) the technique extends to nonlinear problems, as shown in the
main body of the paper.

Our result states that if the parameters $\theta_{k}$ have small norm
$\Vert\theta\Vert_{\ell^{\infty}}$, the solution of the stochastic problem
is close to the solution of the heat equation; and if the intensity $\kappa$ of
the noise is large, the solution of the heat equation decays fast to zero, so
does the solution of the stochastic problem on a finite time interval. In a
sense, with this result the theory of mixing meets the theory of turbulent
diffusion, see \cite{MajdaKramer}.

\begin{remark}
The case of Kraichnan noise, including the particular case of Kolmogorov 41
scaling, is included in the previous example and it may be useful to see the
meaning of our conditions on $\kappa$ and $\Vert\theta\Vert_{\ell^{\infty}}$
in such a case. The divergence free part of Kraichnan noise, on the $d$-dimensional torus
$\mathbb{T}^{d}$, is usually defined by means of the covariance,
space-homogeneous, given by the matrix-function%
\[
Q\left(  z\right)  =\sigma^{2}\sum_{\left\vert k\right\vert \geq k_{0}}%
\frac{k_{0}^{\zeta}}{\left\vert k\right\vert ^{d+\zeta}} \bigg({\rm Id} -\frac{k\otimes
k}{\left\vert k\right\vert ^{2}} \bigg) e^{ik\cdot z}, %
\]
where ${\rm Id}$ is the identity matrix, $k_{0}$ is a positive number and the sum is computed over all
$k\in\mathbb{Z}_0^{d}$ such that $|k| \geq k_{0}$. Kolmogorov 41 scaling is given by the value
$\zeta=4/3$.

The covariance function of the noise $W(t,x)$ defined in \eqref{eq:noise-higherdim}, which
is space-homogeneous, is given by
\begin{align*}
Q_{W}\left( z\right)   &  =\mathbb{E}\left[  W(1,x+z)\otimes W(1,x)\right]
=2\kappa \sum_{k\in \Z_0^d} \sum_{i=1}^{d-1} \theta_{k}^{2}\,\sigma_{k,i}(x+z)\otimes
\sigma_{-k,i}(x)\\
&  =2\kappa \sum_{k\in \Z_0^d} \theta_{k}^{2} \bigg(\sum_{i=1}^{d-1} a_{k,i}\otimes a_{k,i} \bigg)e^{ik\cdot z} .
\end{align*}
Note that
  $${\rm Id} -\frac{k\otimes k}{\left\vert k\right\vert ^{2}} = \sum_{i=1}^{d-1} a_{k,i}\otimes a_{k,i}, \quad k\in \Z^d_0, $$
hence the comparison with Kraichnan noise and
our noise is%
\begin{align*}
2\kappa =\sigma^{2},\quad
\theta_{k}^{2}  =\frac{k_{0}^{\zeta}}{\left\vert k\right\vert ^{d+\zeta}%
} {\bf 1}_{\left\{  \left\vert k\right\vert \geq k_{0}\right\}  }.
\end{align*}
We then have
\[
\Vert\theta\Vert_{\ell^{\infty}} =k_{0}^{-d/2}%
\]
and therefore the condition that $\Vert\theta\Vert_{\ell^{\infty}}$ is
small corresponds to the requirement that $k_{0}$ is large, namely that we
consider a noise acting at small scales. Simultaneously we ask that
$\sigma^{2}$ is large. Therefore we can prove a finite-horizon mixing property
using a small scale large intensity Kraichnan-type noise.
\end{remark}

\subsubsection{Dissipation enhancement}\label{subs-dissip-enhanc}

Similarly to Section \ref{subsec-mixing}, we consider on $\mathbb{T}^{d}$ the
stochastic transport equation but now \textit{with dissipation}%
\begin{equation}\label{heat transport eq}
\d f+\circ\d W\cdot\nabla f=\nu\Delta f\,\d t%
\end{equation}
where the noise $W$ is the same as in the previous section,
parametrized by the pair $(\kappa,\theta)$. The quantitative estimates
developed in this paper allows us, in this particular case, to go beyond a
result of interest over finite time and prove the following result on decay
at infinity (see Section \ref{subs-dissip-enhanc-proof} for its proof).

\begin{theorem}\label{thm-transp-diffus}
For any $p\geq 1$ and $\lambda>0$, there exists a pair $(\kappa,\theta)$ with the following
property: for every $f_{0}\in L^{2}(  \mathbb{T}^{d})$ with zero mean, there
exists a random constant $C>0$ with finite $p$-th moment, such that for the solution $f_{t}$ of equation
\eqref{heat transport eq} with initial condition $f_{0}$, we have $\mathbb{P}%
$-a.s.
\[
\| f_{t} \|_{L^2} \leq C e^{-\lambda t} \|f_0 \|_{L^2}\quad \mbox{for all } t\geq0.
\]
\end{theorem}

Results of this form have been obtained in \cite{BedroBlum}; those results
are technically more demanding, in particular because the noise $W$ is not
white in time, and to some extent it is very general and not required to have
parameters (like $(\kappa,\theta)$) large or small, namely close to their
scaling limit.
After the first version of this paper was completed, the preprint \cite{GessYar} appeared, where the authors readapt the techniques from \cite{Dolgo, BedroBlum} to establish long time mixing and enhanced dissipation estimates for solutions to \eqref{heat transport eq}.
The results therein are far reaching and allow for very low dimensional noise, provided it is smooth enough; in comparison, our techniques have the advantage that they allow rougher noise and that they work for both linear and nonlinear PDEs.
%In a sense, our result is more reminiscent of the finite dimensional theorem of stabilization by noise \cite{arnold}.

\subsection{Our strategy}

\label{subsec-heuristic}

For the reader's convenience, we give here a brief description of our approach
on how to derive quantitative estimates in the style of Theorem \ref{thm-1}.
As the method is quite general, we present it in an abstract but rather simple
setting, in order to highlight the main ideas. The strategy for proving
Theorem \ref{thm-3} follows similar considerations but will not be discussed here.

Consider an SPDE with a nonlinearity $F$ of the form
\begin{equation}
\label{eq:heuristic-eq1}\mathrm{d} \omega+ F(\omega)\,\mathrm{d} t +
\circ\mathrm{d} W\cdot\nabla\omega= \nu\Delta\omega\, \mathrm{d} t
\end{equation}
where $W$ is defined as above for a given pair $(\kappa,\theta)$. As mentioned
above, the structure of the noise is so that the SPDE has equivalent It\^o
formulation
\[
\mathrm{d} \omega+ F(\omega)\,\mathrm{d} t + \mathrm{d} W\cdot\nabla\omega=
(\kappa+\nu) \Delta\omega\, \mathrm{d} t.
\]
Such formulation can be misleading, as the enhanced viscosity $\kappa\Delta$
does not imply any regularizing effect at this stage; the right way to derive
an energy balance for the solution is still to use the Stratonovich
formulation \eqref{eq:heuristic-eq1}. Nevertheless, if suitable a priori
estimates for $\omega$ are available, the stochastic term $\mathrm{d}
M^{\omega}:= \mathrm{d} W\cdot\nabla\omega$ is a well defined martingale
(taking values in a suitable distributional space); the SPDE can therefore be
written as
\begin{equation}
\label{eq:heuristic-eq2}\mathrm{d} \omega+ F(\omega)\,\mathrm{d} t =
(\kappa+\nu)\Delta\omega\, \mathrm{d} t - \mathrm{d} M^{\omega}.
\end{equation}
In particular, equation \eqref{eq:heuristic-eq2} can be regarded as a
stochastic perturbation of the deterministic PDE with enhanced viscosity
\begin{equation}
\label{eq:heuristic-eq3}\partial_{t} \tilde{\omega} + F(\tilde{\omega}) =
(\kappa+\nu)\Delta\tilde{\omega}%
\end{equation}
due to the presence of a stochastic forcing term $\mathrm{d} M^{\omega}$; the
key point is that the very poor space-time regularity of $\mathrm{d}
M^{\omega}$ exactly counters the term $(\kappa+\nu)\Delta$ and does not allow
to derive estimates for $\omega$ in $H^{\alpha}$ for any $\alpha\geq0$. This
is consistent with the fact that the \textit{variational approach} requires to
interpret the SPDE in the Stratonovich form \eqref{eq:heuristic-eq1}, as done
in \cite{FGL2}.

The major intuition of the current work is that, while formulation
\eqref{eq:heuristic-eq2} cannot be used to derive estimates in strong spaces,
it can be employed within the \textit{semigroup approach} to obtain rates of
convergence in the weaker scales $H^{-\alpha}$ for $\alpha>0$.

To explain what we mean, we start by writing equation \eqref{eq:heuristic-eq2}
in the corresponding mild formulation
\[
\omega_{t} = P_{t} \omega_{0} - \int_{0}^{t} P_{t-s} F(\omega_{s})\,
\mathrm{d} s - Z_{t}, \quad Z_{t} := \int_{0}^{t} P_{t-s}\, \mathrm{d}
M^{\omega}_{s}.
\]
Here $P_{t} = e^{t(\kappa+\nu)\Delta}$ for $t\geq0$, while the process $Z$ is
an instance of a \textit{stochastic convolution}. Both the passage from weak
to mild formulation and the definition of $Z$ are classical, but will be
explained in detail in Sections \ref{subsec-stoch-convol} and
\ref{sec:mild-form}.

Consider now a solution $\tilde{\omega}$ to \eqref{eq:heuristic-eq3} with the
same initial data $\omega_{0}$ and write it in mild form (which is the same as
above with $Z\equiv0$). Defining the difference $\xi_{t} = \omega_{t}%
-\tilde{\omega}_{t}$, it holds
\begin{equation}
\label{eq:heuristic-eq4}\xi_{t} = -\int_{0}^{t} P_{t-s} \big[F(\omega_{s})
-F(\tilde\omega_{s})\big]\,\mathrm{d} s - Z_{t}.
\end{equation}
At this stage, the stochastic process $Z$ can be seen as a random element of
$C([0,T];H^{-\alpha})$ and equation \eqref{eq:heuristic-eq4} can be treated by
purely analytic methods in a pathwise manner. Assume that the nonlinearity $F$
satisfies some regularity condition of the form
\begin{equation}
\label{eq:heuristic-eq5}\| F(\omega)-F(\tilde{\omega})\|_{H^{-\alpha-1}}
\lesssim\| \omega-\tilde{\omega}\|_{H^{-\alpha}} \quad\forall\, \omega
,\tilde{\omega}\in L^{2};
\end{equation}
such assumption is not very realistic, especially for polynomial
nonlinearities, and more complicated variants should be considered, but for
the sake of exposition here we stick to \eqref{eq:heuristic-eq5}.

Assumption \eqref{eq:heuristic-eq5}, together with classical estimates for
convolution with heat kernel (which will be recalled in Section
\ref{subsec:technical-lemmas}), imply that the solution $\xi$ to
\eqref{eq:heuristic-eq4} satisfies
\[
\aligned
\| \xi_{t}\|_{H^{-\alpha}}^{2}  &  \lesssim\frac1{\kappa+\nu} \int_{0}^{t} \|
F(\omega_{s})-F(\tilde{\omega}_{s})\|_{H^{-\alpha-1}}^{2}\, \mathrm{d} s +
\|Z_{t}\|_{H^{-\alpha}}^{2} \newline &  \lesssim\frac1{\kappa+\nu} \int%
_{0}^{t} \| \xi_{s}\|_{H^{-\alpha}}^{2}\, \mathrm{d} s + \| Z_{t}%
\|_{H^{-\alpha}}^{2}. \endaligned
\]
Then, an application of Gronwall's lemma yields the pathwise estimate
\begin{equation}
\label{eq:heuristic-eq6}\sup_{t\in[0,T]} \| \xi_{t}\|_{H^{-\alpha}} \lesssim
e^{CT} \sup_{t\in[0,T]} \| Z_{t}\|_{H^{-\alpha}}%
\end{equation}
for some constant $C$ depending on parameters like $\alpha, \kappa, \nu$, etc.
Finally, estimates like those in Theorem \ref{thm-1} will follow from taking
expectation in \eqref{eq:heuristic-eq6}, assuming we have enough integrability
on $Z$ and we can control it in terms of $\| \theta\|_{\ell^{\infty}}$; this
is indeed possible and will be presented in Section \ref{subsec-stoch-convol}.

Overall we see that obtaining convergence rates of the SPDE
\eqref{eq:heuristic-eq1} to the deterministic PDE \eqref{eq:heuristic-eq3}
requires a nice interplay of analytic and probabilistic arguments as follows:

\begin{itemize}
\item Passing from \eqref{eq:heuristic-eq4} to \eqref{eq:heuristic-eq6} in a
pathwise manner requires an assumption on the nonlinearity $F$ similar (but
possibly more complicated) to \eqref{eq:heuristic-eq5}; this step is purely
analytical and requires different treatment depending on the PDE in consideration.

\item Estimating the stochastic convolution $Z$ instead can be done in full
generality and mostly relies on probabilistic arguments; it requires however
some information on $M^{\omega}$ and subsequently the given solution $\omega$.

\item The Stratonovich formulation \eqref{eq:heuristic-eq1}, together with the
divergence free structure of the noise and the variational approach, are the
right tools to derive a priori estimates on $\omega$; this step is a mixture
of analytic and probabilistic techniques and has already been developed for
several PDEs in \cite{Gal, FGL, Luo, LuoSaal, FGL2}.
\end{itemize}

We conclude the introduction with the structure of the paper.
In Section 2 we make some preparations, the most important part being devoted to maximal estimates for stochastic convolutions.
Theorem \ref{thm-1} will be proved in Section \ref{sec:euler}, together with its extensions to 2D Boussinesq and mSQG systems.
Section \ref{sec:blow-up} is dedicated to the estimates on blow-up probability, dealing with the Keller-Segel model. We provide the proofs of Theorems \ref{thm-transport} and \ref{thm-transp-diffus} in the two subsections of Section \ref{sec:transport}; some related results concerning the solution operator of \eqref{stoch-transp-Ito} will also be proved in Section \ref{subs-inviscid-transport}.

\section{Preliminaries}\label{sec:preliminaries}

In this section we provide several tools of fundamental importance for the
next sections: Section \ref{subsec:technical-lemmas} contains some technical
results that will be frequently used below; Section \ref{subsec-stoch-convol}
presents maximal estimates for stochastic convolutions; finally Section
\ref{sec:mild-form} explains the link between weak and mild form for the class
of SPDEs of our interest.

\subsection{Some analytical lemmas}

\label{subsec:technical-lemmas}

We first recall the following well known estimates on the transport term and
products of functions in Sobolev spaces.

\begin{lemma}
\label{lem-estimate} In the following, $V\in L^{2}(\mathbb{T}^{2}%
,\mathbb{R}^{2})$ always denotes a divergence free vector field.

\begin{itemize}
\item[\textrm{(a)}] For $V\in L^{\infty}(\mathbb{T}^{2},\mathbb{R}^{2})$ and
$f\in L^{2}(\mathbb{T}^{2})$, we have
\[
\|V\cdot\nabla f\|_{H^{-1}} \lesssim\|V\|_{L^{\infty}} \|f\|_{L^{2}}.
\]

\item[\textrm{(b)}] Let $\alpha\in(1,2]$, $\beta\in(0,\alpha-1)$, $V\in
H^{\alpha}(\mathbb{T}^{2},\mathbb{R}^{2})$ and $f\in H^{-\beta}(\mathbb{T}%
^{2})$, then
\[
\|V\cdot\nabla f\|_{H^{-1-\beta}} \lesssim_{\alpha,\beta} \|V\|_{H^{\alpha}}
\|f\|_{H^{-\beta}}.
\]

\item[\textrm{(c)}] Let $\beta\in(0,1)$, then for any $f\in H^{\beta
}(\mathbb{T}^{2})$ and $g\in H^{1-\beta}(\mathbb{T}^{2})$ it holds
\[
\|f\, g\|_{L^{2}} \lesssim_{\beta}\|f\|_{H^{\beta}} \|g\|_{H^{1-\beta}}.
\]

\item[\textrm{(d)}] Let $\beta\in(0,1)$, $V\in H^{1-\beta}(\mathbb{T}%
^{2},\mathbb{R}^{2})$ and $f\in L^{2}(\mathbb{T}^{2})$, then one has
\[
\| V\cdot\nabla f\|_{H^{-1-\beta}} \lesssim_{\beta}\| V\|_{H^{1-\beta}} \, \|
f\|_{L^{2}}.
\]
\end{itemize}
\end{lemma}

\begin{proof} The proofs are classical, but we provide them for completeness. First observe that, by the divergence free assumption, $\| V\cdot\nabla f\|_{H^{s-1}} = \| \nabla\cdot (Vf)\|_{H^{s-1}} \lesssim \| V f\|_{H^s}$ for any $s\in\R$. Point (a) then immediately follows from $\| Vf\|_{L^2} \leq \| V\|_{L^\infty} \| f\|_{L^2}$.

(b) By the assumptions and Sobolev embedding, $V\in C^{s}(\T^2,\R^2)$ for any $s\in (\beta,\alpha-1)$; by classical results regarding paraproducts (see for instance \cite{BCD}), the product between $f\in H^{-\beta}$ and $V\in C^s$ with $s>\beta$ is a well defined element of $H^{-\beta}$ and $\| fV\|_{H^{-\beta}} \lesssim \| V\|_{C^s} \| f\|_{H^{-\beta}}$.

(c) The assertion follows from H\"older's inequality combined with the Sobolev embeddings $\|f\|_{L^{2/(1-\beta)}} \lesssim \|f\|_{H^{\beta}} $, $\|g\|_{L^{2/\beta}} \lesssim \|f\|_{H^{1-\beta}} $.

(d) By point (c), for any $\varphi\in C^\infty(\T^2,\R^2)$ we have
$$\aligned
|\<V\, f, \varphi\>| &= |\<f, V\cdot \varphi\>| \leq \|f\|_{L^2} \|V\cdot \varphi\|_{L^2}
\lesssim \|f\|_{L^2} \|V\|_{H^{1-\beta}} \| \varphi\|_{H^\beta}
\endaligned $$
showing that $f V\in H^{-\beta}(\T^2,\R^2)$; the desired estimate follows from the initial observation.
\end{proof}

Next we state some classical heat kernel estimates for later use.

\begin{lemma}\label{lem:heat-kernel-estim}
Let $u\in H^{\alpha}$, $\alpha\in\mathbb{R}$. Then:

\begin{itemize}
\item[\textrm{(i)}] for any $\rho\geq0$, it holds $\|e^{t\Delta} u
\|_{H^{\alpha+\rho}} \leq C_{\rho}\, t^{-\rho/2} \|u\|_{H^{\alpha}}$ for some
constant increasing in $\rho$;

\item[\textrm{(ii)}] for any $\rho\in[0,2]$, it holds $\| (I-e^{t\Delta})
u\|_{H^{\alpha-\rho}}\lesssim t^{\rho/2} \|u\|_{H^{\alpha}}$.
\end{itemize}
\end{lemma}

We also present the following regularizing effect by convolution with
$e^{\delta t \Delta}$ for some $\delta>0$.

\begin{lemma}
\label{lem:heat-kernel} For any $\alpha\in\mathbb{R}$ and any $f\in
L^{2}(0,T;H^{\alpha})$, it holds
\[
\bigg\|\int_{0}^{t} e^{\delta(t-s) \Delta} f_{s}\, \mathrm{d} s
\bigg\|_{H^{\alpha+1}}^{2} \lesssim\frac1{\delta} \int_{0}^{t} \|
f_{s}\|_{H^{\alpha}}^{2}\, \mathrm{d} s \quad\forall\, t\in[0,T].
\]

\end{lemma}

\begin{proof}
For any fixed $t\in [0,T]$, it holds
\begin{equation*}\begin{split}
\bigg\|\int_0^t e^{\delta(t-s) \Delta} f_s\, \d s \bigg\|_{H^{\alpha+1}}^2
& = \sum_k |k|^{2(\alpha+1)} \bigg| \int_0^t e^{-4\pi^2 \delta(t-s)|k|^2} \langle f_s,e_k\rangle\, \d s\bigg|^2\\
& \leq \sum_k |k|^{2(\alpha+1)} \int_0^t e^{-8\pi^2 \delta(t-s)|k|^2}\,\d s \, \int_0^t |\langle f_s,e_k\rangle|^2\, \d s\\
& \lesssim \frac1{\delta} \sum_k |k|^{2\alpha} \int_0^t |\langle f_s,e_k\rangle|^2\, \d s
\end{split}\end{equation*}
which gives the conclusion.
\end{proof}

\subsection{Maximal estimates on stochastic convolution}

\label{subsec-stoch-convol}

We present here some estimates for the process $Z$ which was shortly
introduced in Section \ref{subsec-heuristic}. Maximal estimates for stochastic
convolutions are not a new topic, see \cite{Kot, Tub, DPZ92b} for some
classical results; however we have not found in the literature a result
fitting our framework, which is why we provide it here.

For future use, we assume we are in $\mathbb{T}^{d}$ with $d\geq2$ and
consider the noise $W$ defined in \eqref{eq:noise-higherdim}.
As the definition of $Z$ is independent of the specific SPDE in consideration,
we pose ourselves in a slightly more general framework. Throughout this
section we will assume $\omega$ is just a given $L^{2}$-valued stochastic
process with measurable trajectories satisfying the following

\begin{assumption}
\label{ass:bddness} There exists a deterministic constant $R>0$ such that
\[
\sup_{t\geq0} \| \omega_{t}\|_{L^{2}} \leq R\quad\mathbb{P}\mbox{-a.s.}
\]

\end{assumption}

Given an Hilbert space $E$, we will denote by $[M]_{E}$ the cross-quadratic
variation of an $E$-valued martingale $M$, namely the unique increasing
process such that $\| M\|_{E}^{2} - [M]_{E}$ is a real-valued martingale (so
the definition depends on the choice of $\| \cdot\|_{E}$). We mention that
Burkholder-Davis-Gundy's inequality still holds on Hilbert spaces (and in the
more general class of UMD Banach spaces, cf. \cite{NVM07}).

For $W$ as in \eqref{eq:noise-higherdim} and a stochastic process $\omega$ satisfying Assumption
\ref{ass:bddness}, we define
\[
M_{t} = \int_{0}^{t} \nabla\omega_{s}\cdot\mathrm{d} W_{s} = \sqrt{C_{d}
\kappa} \int_{0}^{t} \sum_{k,i} \theta_{k} \sigma_{k,i}\cdot\nabla\omega_{s}\,
\mathrm{d} W^{k,i}_{s}
\]
where we simply write $\sum_{k,i}$ in place of $\sum_{k\in\mathbb{Z}^{d}_{0}}\sum_{i=1}^{d-1}$.
$M$ is a well-defined continuous martingale with values in $H^{-1}$; indeed,
\begin{align*}
\mathbb{E}\bigg[\sup_{t\in[0,T]} \Big\| \int_{0}^{t} \mathrm{d} W_{s}%
\cdot\nabla\omega_{s}\Big\|_{H^{-1}}^{2}\bigg]  &  \lesssim\kappa\,
\mathbb{E}\bigg[ \sum_{k,i} \int_{0}^{T} \theta_{k}^{2} \| \sigma_{k,i}
\cdot\nabla\omega_{s}\|_{H^{-1}}^{2}\, \mathrm{d} s\bigg]\\
&  \lesssim\kappa\sum_{k,i} \int_{0}^{T} \theta_{k}^{2}\, \mathbb{E}\big[ \|
a_{k,i} e_{k}\|_{L^{\infty}}^{2} \| \omega_{s}\|_{L^{2}}^{2}\big]\, \mathrm{d}
s\\
&  \lesssim\kappa\, \|\theta\|_{\ell^{2}}^{2}\, R^{2}\, T <\infty,
\end{align*}
where we used the property $\big[W^{k,i}, W^{l,j} \big]_{t}= 2t \delta
_{k,-l}\delta_{i,j}$, Lemma \ref{lem-estimate}(a) (holds also in high
dimensions) and Assumption \ref{ass:bddness}.

Given $\delta>0$, our aim is to study the stochastic convolution process
$\{Z_{t}\}_{t\in[0,T]}$ given by
\begin{equation}
\label{eq:stoch-conv}Z_{t} = \int_{0}^{t} e^{\delta(t-s)\Delta}\, \mathrm{d}
M_{s} = \sqrt{C_{d} \kappa} \int_{0}^{t} \sum_{k,i} \theta_{k} e^{\delta
(t-s)\Delta} (\sigma_{k,i}\cdot\nabla\omega_{s})\, \mathrm{d} W^{k,i}_{s}.
\end{equation}

\begin{lemma}
\label{lem:estim-Z} Let $\kappa,\delta>0$, $\theta\in\ell^{2}$ as above,
$\omega$ satisfying Assumption \ref{ass:bddness} and define $Z$ as in
\eqref{eq:stoch-conv}. Then for any $\varepsilon\in(0,1/2)$ and any
$p\in[1,\infty)$ it holds
\begin{equation}\label{eq:estim-Z-ell2}
\mathbb{E} \bigg[ \sup_{t\in[0,T]} \| Z_{t}%
\|_{H^{-\varepsilon}}^{p} \bigg]^{1/p} \lesssim_{\varepsilon,p,T} \sqrt
{\kappa\delta^{\varepsilon-1}}\, \|\theta\|_{\ell^{2}} R;
\end{equation}
similarly,
\begin{equation}\label{eq:estim-Z-ellinfty}
\mathbb{E} \bigg[ \sup_{t\in[0,T]} \|
Z_{t}\|_{H^{-d/2-\varepsilon}}^{p} \bigg]^{1/p} \lesssim_{\varepsilon,p,T}
\sqrt{\kappa\delta^{\varepsilon-1}} \, \|\theta\|_{\ell^{\infty}} R.
\end{equation}
\end{lemma}

\begin{proof}
For fixed $\eps\in (0,1/2]$ and $t\in [0,T]$, we can estimate $\| Z_t\|_{H^{-\eps}}$ by Burkholder-Davis-Gundy's inequality:
\begin{align*}
\E\big[ \| Z_t\|_{H^{-\eps}}^{2p} \big]^{1/2p}
& \sim \sqrt{\kappa}\, \E \bigg[ \Big\| \sum_{k,i} \theta_k \int_0^t e^{\delta(t-r)\Delta} ( \sigma_{k,i}\cdot\nabla \omega_r)\, \d W^{k,i}_r\Big\|^{2p}_{H^{-\eps}} \bigg]^{1/2p}\\
& \lesssim_p \sqrt{\kappa}\, \E \bigg[ \Big[ \sum_{k,i} \theta_k \int_0^\cdot e^{\delta(t-r)\Delta} ( \sigma_{k,i}\cdot\nabla \omega_r)\, \d W^{k,i}_r \Big]^p_{t;H^{-\eps}} \bigg]^{1/2p}\\
& \lesssim \sqrt{\kappa}\, \E \bigg[ \bigg( \sum_{k,i} \theta_k^2 \int_0^t \big\| e^{\delta(t-r)\Delta} ( \sigma_{k,i}\cdot\nabla \omega_r) \big\|_{H^{-\eps}}^2\, \d r\bigg)^p\bigg]^{1/2p}.
\end{align*}
Next, we apply Lemma \ref{lem:heat-kernel-estim}(i) with $\rho= 1-\eps$ (since $\eps>0$ it holds $C_\rho= C_{1-\eps}\leq C_1$) and obtain
\begin{align*}
\E\big[ \| Z_t\|_{H^{-\eps}}^{2p} \big]^{1/2p}
& \lesssim \sqrt{\kappa \delta^{\eps -1}}\, \E \bigg[ \bigg( \sum_{k,i} \theta_k^2 \int_0^t |t-r|^{\eps-1} \| \sigma_{k,i}\cdot\nabla \omega_r \|_{H^{-1}}^2\, \d r\bigg)^p\bigg]^{1/2p}\\
& \lesssim \sqrt{\kappa \delta^{\eps -1}}\,\|\theta\|_{\ell^2}\, \E \Big[\,\| \omega\|_{L^\infty(0,T; L^2)}^{2p} \Big]^{1/2p} \bigg( \int_0^t |t-r|^{\eps-1}\, \d r\bigg)^{1/2}\\
& \lesssim_{T} \sqrt{\kappa \delta^{\eps -1} \eps^{-1}}\, \| \theta\|_{\ell^2} R,
\end{align*}
where the last two steps follow from Lemma \ref{lem-estimate}(a) with $\|\sigma_{k,i}\|_{L^\infty}=1$ and Assumption \ref{ass:bddness}. A similar computation shows that
\begin{align*}
\sqrt{\kappa}\, \E \bigg[ \Big\| \int_s^t e^{\delta(t-r)\Delta}\, \d M_r\Big\|^{2p}_{H^{-\eps}} \bigg]^{1/2p} \lesssim_{p} \sqrt{\kappa \delta^{\eps -1} \eps^{-1}}\, |t-s|^{\eps/2} \| \theta\|_{\ell^2} R.
\end{align*}
Next, observing that by construction $Z$ satisfies the relation
\begin{equation*}%\label{lem:estim-Z.1}
Z_t = e^{\delta(t-s)\Delta} Z_s + \int_s^t e^{\delta(t-r)\Delta}\, \d M_r,
\end{equation*}
by Lemma \ref{lem:heat-kernel-estim}(ii) we obtain
\begin{align*}
\|Z_t-Z_s\|_{H^{-2\eps}}
& \leq \| (I-e^{\delta(t-s)\Delta}) Z_s\|_{H^{-2\eps}} + \Big\| \int_s^t e^{\delta(t-r)\Delta}\, \d M_r\Big\|_{H^{-2\eps}}\\
& \lesssim \delta^{\eps/2} |t-s|^{\eps/2} \| Z_s\|_{H^{-\eps}} + \Big\| \int_s^t e^{\delta(t-r)\Delta}\, \d M_r\Big\|_{H^{-2\eps}}.
\end{align*}
Taking expectation and applying the previous estimates we arrive at
\begin{align*}
\E\big[ \| Z_t-Z_s\|_{H^{-2\eps}}^{2p}\big]^{1/2p}
& \lesssim_{p,T} \sqrt{\kappa \delta^{2\eps -1} \eps^{-1}}\, \|\theta\|_{\ell^2} R \big(|t-s|^{\eps/2} + |t-s|^\eps\big)\\
& \lesssim_{p,T} \sqrt{\kappa \delta^{2\eps -1} \eps^{-1}}\, \|\theta\|_{\ell^2} R |t-s|^{\eps/2}.
\end{align*}
Renaming $2\eps$ as $\eps$ gives us
$$\E\big[ \| Z_t-Z_s\|_{H^{-\eps}}^{2p}\big] \lesssim_{p,T}  \big(\sqrt{\kappa \delta^{\eps -1} \eps^{-1}}\,  \|\theta\|_{\ell^2} R \big)^{2p} |t-s|^{p\eps/2} . $$
Now for $\eps\in (0, 1/2)$, choosing $p > 2/\eps$ (which is allowed since otherwise we can control the $L^p$-norm by the $L^{\tilde{p}}$-one for some $\tilde{p}>p$) and applying Kolmogorov's continuity criterion (which produces some additional constants depending on $p,\eps$) we obtain \eqref{eq:estim-Z-ell2}.

The proof of \eqref{eq:estim-Z-ellinfty} is very similar, so we only sketch it. Repeating the initial computations with $H^{-d/2-2\eps}$ in place of $H^{-\eps}$, we arrive at
\begin{align*}
\E\Big[ \| Z_t\|_{H^{-d/2-2\eps}}^{2p} \Big]^{1/2p}
& \lesssim_p \sqrt{\kappa}\, \E \bigg[ \bigg( \sum_{k,i} \theta_k^2 \int_0^t \big\| e^{\delta(t-r)\Delta} ( \sigma_{k,i}\cdot\nabla \omega_r) \big\|_{H^{-d/2-2\eps}}^2 \d r\bigg)^p\bigg]^{1/2p}\\
& \lesssim \sqrt{\kappa \delta^{\eps-1}}\, \|\theta \|_{\ell^\infty} \E \bigg[\bigg( \int_0^t |t-r|^{\eps-1} \sum_{k,i}  \| \sigma_{k,i}\cdot\nabla \omega_r \|_{H^{-1-d/2-\eps}}^2 \d r\bigg)^p\bigg]^{1/2p}
\end{align*}
where in the last step we used again Lemma \ref{lem:heat-kernel-estim}(i) with $\rho=1-\eps$. We have
$$\| \sigma_{k,i}\cdot\nabla \omega_r\|_{H^{-1-d/2-\eps}} = \| \nabla\cdot (\sigma_{k,i}\, \omega_r) \|_{H^{-1-d/2-\eps}} \lesssim \| \sigma_{k,i}\, \omega_r \|_{H^{-d/2-\eps}} \lesssim \| e_k\, \omega_r \|_{H^{-d/2-\eps}}$$
and
$$\sum_k  \|e_k\, \omega_r\|_{H^{-d/2-\eps}}^2 \lesssim \sum_k  \sum_{l}\frac1{|l|^{d+2\eps}} |\<\omega_r, e_{l-k}\>|^2 = \|\omega_r \|_{L^2}^2 \sum_{l} \frac1{|l|^{d+2\eps}} \lesssim \eps^{-1} \|\omega_r \|_{L^2}^2.  $$
The last step is due to
$$\sum_{l} \frac1{|l|^{d+2\eps}} \leq \int_{\{x\in\R^d: |x|\geq 1/2\}} \frac{\d x}{|x|^{d+2\eps}} \sim \int_{1/2}^\infty \frac{\d s}{s^{1+2\eps}} \sim \eps^{-1}. $$
Combining these estimates with Assumption \ref{ass:bddness} yields
\begin{align*}
\E\Big[ \| Z_t\|_{H^{-d/2-2\eps}}^{2p} \Big]^{1/2p} & \lesssim \sqrt{\kappa \delta^{\eps-1} \eps^{-1}}\, \|\theta \|_{\ell^\infty} R \bigg(\int_0^t |t-r|^{\eps-1}\d r \bigg)^{1/2} \\
&\lesssim_T \sqrt{\kappa \delta^{\eps-1} \eps^{-2}}\, \|\theta \|_{\ell^\infty} R.
\end{align*}
From here on, the proof is almost identical to the one of \eqref{eq:estim-Z-ell2}.
\end{proof}

\begin{corollary}
\label{cor-stoch-convol} Suppose now $\| \theta\|_{\ell^{2}}=1$. Then for any
$\beta\in(0,d/2]$ and any $\varepsilon\in(0,\beta]$, it holds
\begin{equation}
\label{eq:estim-Z-interp.0}\mathbb{E} \bigg[ \sup_{t\in[0,T]} \|
Z_{t}\|_{H^{-\beta}}^{p}\bigg]^{1/p} \lesssim_{\varepsilon,p,T} \sqrt
{\kappa\delta^{\varepsilon-1}} \, \|\theta\|_{\ell^{\infty}}^{2(\beta
-\varepsilon)/d} R.
\end{equation}
In particular, if $d=2$, then for any $\beta\in(0,1]$ and any $\varepsilon
\in(0,\beta]$, it holds
\begin{equation}
\label{eq:estim-Z-interp}\mathbb{E} \bigg[ \sup_{t\in[0,T]} \| Z_{t}%
\|_{H^{-\beta}}^{p}\bigg]^{1/p} \lesssim_{\varepsilon,p,T} \sqrt{\kappa
\delta^{\varepsilon-1}} \, \|\theta\|_{\ell^{\infty}}^{\beta-\varepsilon} R.
\end{equation}

\end{corollary}

\begin{proof}
Setting $\lambda=2(\beta-\eps)/d$, then $\lambda\in [0,1)$ and $-\beta= \lambda(-d/2-\eps)+ (1-\lambda)(-\eps)$. Estimate \eqref{eq:estim-Z-interp.0} readily follows by interpolating between \eqref{eq:estim-Z-ell2} and \eqref{eq:estim-Z-ellinfty}.
\end{proof}

\subsection{Mild formulation of the equation}

\label{sec:mild-form}

In this section we rigorously show that weak solutions of a general class of
SPDEs also satisfy the corresponding mild formulation. The result is
classical, cf. \cite[Theorem 6.5]{DPZ92}, but rather unusual for SPDEs of
hyperbolic nature (think of stochastic 2D Euler, namely eq. \eqref{stoch-NS}
with $\nu=0$), which is why we prefer to present the argument with some care.

We consider a class of SPDEs of the form
\begin{equation}
\label{eq:SPDE-stratonovich}\mathrm{d} \omega_{t} = [\nu\Delta\omega_{t} +
F(\omega_{t})]\, \mathrm{d} t + \circ\mathrm{d} W_{t}\cdot\nabla\omega_{t}.
\end{equation}
For simplicity we assume we are on the torus $\mathbb{T}^{2}$ with periodic
boundary condition and noise defined as in \eqref{noise} for a given pair
$(\kappa,\theta)$; the generalization to higher dimension $d\geq3$ or
different domains $\Omega\subset\mathbb{R}^{d}$ is omitted here. By standard
computations (see \cite{Gal, FGL2} for more details) we can rewrite the above
SPDE in the equivalent It\^o form (we set $\delta=\kappa+\nu$ for notational
simplicity)
\begin{equation}
\label{eq:SPDE-ito}\mathrm{d} \omega_{t} = [\delta\Delta\omega_{t} +
F(\omega_{t})]\,\mathrm{d} t + \mathrm{d} W_{t}\cdot\nabla\omega_{t}.
\end{equation}

We impose the following assumption on the nonlinearity $F$: there exists
$s\geq0$ big enough and an increasing function $G:\mathbb{R}_{+} \to
\mathbb{R}_{+}$ such that $F$ maps $L^{2}$ into $H^{-s}$ and satisfies
\[
\| F(\omega)\|_{H^{-s}} \leq G(\| \omega\|_{L^{2}})\quad\forall\, \omega\in
L^{2}.
\]
It is immediate to verify that the nonlinearities associated to the Euler,
mSQG and Keller-Segel equations satisfy this condition.

\begin{definition}
\label{defn:weak-sol} Let $(\Omega,\mathcal{F},\{\mathcal{F}_{t}\}_{t\geq0},
\mathbb{P})$ be a complete filtered probability space on which $W$ is defined
by \eqref{noise} (i.e. $W^{k}$ are $\mathcal{F}_{t}$-Brownian motions); let
$\omega$ be an $L^{2}$-valued, $\mathcal{F}_{t}$-adapted stochastic process
satisfying Assumption \ref{ass:bddness}. We say that $\omega$ is a weak
solution to \eqref{eq:SPDE-stratonovich} with initial data $\omega_{0}\in
L^{2}$ if for any $\varphi\in C^{\infty}(\mathbb{T}^{2})$, $\mathbb{P}$-a.s.
it holds, for all $t\in[0,T]$,
\begin{align*}
\langle\omega_{t}, \varphi\rangle= \langle\omega_{0},\varphi\rangle+ \int%
_{0}^{t} \big[\langle\omega_{s}, \delta\Delta\varphi\rangle+ \langle
F(\omega_{s}), \varphi\rangle\big]\, \mathrm{d} s - \sqrt{2\kappa}\sum_{k}
\theta_{k} \int_{0}^{t} \langle\omega_{s}, \sigma_{k}\cdot\nabla\varphi
\rangle\, \mathrm{d} W^{k}_{s} .
\end{align*}

\end{definition}

In the above definition we have imposed for simplicity Assumption
\ref{ass:bddness}, as it fits nicely with the divergence free structure of the
noise and the SPDEs considered here; but the requirement can be further weakened.

Under Assumption \ref{ass:bddness}, the function
\[
t\mapsto\int_{0}^{t} [\delta\Delta\omega_{s} + F(\omega_{s})]\, \mathrm{d} s
\]
is pathwise defined as an element of $C([0,T];H^{-s^{\prime}})$ for
$s^{\prime}=s\vee2$, once we intepret the integral in the Bochner sense.
Indeed we have the estimate
\begin{align*}
\bigg\| \int_{0}^{t} [\delta\Delta\omega_{s} + F(\omega_{s})]\, \mathrm{d}
s\bigg\|_{H^{-s^{\prime}}} \leq\int_{0}^{T} \big[\delta\| \Delta\omega
_{s}\|_{H^{-s^{\prime}}} + \| F(\omega_{s})\|_{H^{-s^{\prime}}} \big]\,
\mathrm{d} s \leq T( \delta R + G(R)).
\end{align*}
Similarly, by the computations from Section \ref{subsec-stoch-convol}, the
process
\[
M_{t} = \int_{0}^{t} \nabla\omega_{s}\cdot\mathrm{d} W_{s}%
\]
is a well-defined, $H^{-1}$-valued, continuous martingale. It is then easy to
check (take a countable collection $\{\varphi_{n}\}_{n\in\mathbb{N}}\subset
C^{\infty}(T^{2})$ which is dense in $H^{s^{\prime}}$) that $\omega$ is a weak
solution in the sense of Definition \ref{defn:weak-sol} if and only if it
satisfies
\[
\omega_{t} = \omega_{0} + \int_{0}^{t} [\delta\Delta\omega_{s}+F(\omega
_{s})]\, \mathrm{d} s + \int_{0}^{t} \nabla\omega_{s}\cdot\mathrm{d} W_{s}
\]
with the integrals being interpreted as above.

\begin{lemma}
\label{lem:weak-implies-mild} Let $\omega$ be a weak solution to the SPDE
\eqref{eq:SPDE-stratonovich} in the sense of Definition \ref{defn:weak-sol},
let $M$ be defined as above and set $P_{t}=e^{\delta t\Delta}$. Then,
$\mathbb{P}$-a.s., it holds
\begin{equation}
\label{eq:mild-form}\omega_{t} = P_{t} \omega_{0} + \int_{0}^{t} P_{t-s}
F(\omega_{s})\, \mathrm{d} s + \int_{0}^{t} P_{t-s}\, \mathrm{d} M_{s}
\quad\forall\, t\in[0,T],
\end{equation}
where the second integral is a stochastic convolution as defined in Section
\ref{subsec-stoch-convol}.
\end{lemma}

\begin{proof}
By definition of the martingale $M$, for any $\varphi\in C^\infty(\T^2)$ it holds
\begin{align*}
\d \langle \varphi, M_t\rangle = \langle \varphi, \d M_t\rangle = - \sqrt{2\kappa} \sum_k \theta_k \langle \omega_t, \sigma_k\cdot\nabla \varphi \rangle\, \d W^k_t.
\end{align*}
For $j\in\Z^2$, set $\lambda_j = 4\pi^2 |j|^2$ and take $\varphi=e_j$ in the definition of weak solution, then
\begin{align*}
\d \langle \omega_t,e_j\rangle
= \big[-\delta\lambda_j\langle \omega_t, e_j\rangle + \langle F(\omega_t), e_j\rangle \big]\, \d t + \langle e_j, \d M_t\rangle.
\end{align*}
Applying the It\^o formula to the process $e^{-t \delta \lambda_j} \langle \omega_t,e_j\rangle$ and integrating in time yield, $\P$-a.s.,
\[
\langle \omega_t ,e_j\rangle = e^{-t \delta\lambda_j} \langle \omega_0,e_j\rangle + \int_0^t e^{-(t-s) \delta\lambda_j}\langle F(\omega_s), e_j\rangle\, \d s + \int_0^t e^{-(t-s) \delta\lambda_j} \langle e_j, \d M_s\rangle \quad \forall\, t\in [0,T].
\]
We can then find $\Gamma\subset \Omega$ of full probability such that the above equality holds for all $t\in [0,T]$ and all $j\in \Z^2$. But this is exactly \eqref{eq:mild-form} written in Fourier modes.
\end{proof}

\section{Proofs of Theorem \ref{thm-1} and related models}\label{sec:euler}

In this section we first prove Theorem \ref{thm-1}, then we
adapt the same idea to treat other fluid dynamical models, including the 2D
Boussinesq system and mSQG equations, which will be presented in Sections
\ref{sec:boussinesq} and \ref{sec:mSQG} respectively.

\subsection{Proofs of Theorem \ref{thm-1} and Corollary \ref{cor-approx-uniq}}\label{subsec-proof}

Let us quickly recall the setting: given $\omega_{0}\in L^{2}$, we consider a
weak solution $\omega$ to the stochastic Euler/Navier-Stokes equation
\eqref{stoch-NS} with the property that
\begin{equation}
\label{L-2-bound}\sup_{t\geq0} \bigg\{ \|\omega_{t}\|_{L^{2}}^{2} + 2\nu
\int_{0}^{t} \| \nabla\omega_{s}\|_{L^{2}}^{2}\, \mathrm{d} s \bigg\}\leq
\|\omega_{0} \|_{L^{2}} \quad\mathbb{P}\mbox{-a.s.}
\end{equation}
For $\nu=0$ weak existence of such solutions follows from \cite[Theorem
2.2]{FGL}, while for $\nu>0$ strong existence and uniqueness is classical (it
can also be derived from the results of \cite{FGL2}). Similarly, we denote by
$\tilde{\omega}$ the solution to the deterministic Navier--Stokes
\eqref{determ-NS} with initial data $\omega_{0}$, which satisfies
\begin{equation}
\label{L-2-bound-determ}\sup_{t\geq0} \bigg\{ \|\tilde\omega_{t}\|_{L^{2}}^{2}
+ 2(\nu+\kappa)\int_{0}^{t} \| \nabla\tilde\omega_{s}\|_{L^{2}}^{2}\,
\mathrm{d} s \bigg\}\leq\|\omega_{0} \|_{L^{2}};
\end{equation}
existence and uniqueness of $\tilde{\omega}$ in the class $L^{2}%
(0,T;H^{1})\cap C([0,T];L^{2})$ is again classical, cf. \cite{Tem} (here
$\nu=0$ does not make any difference due to the presence of $\kappa>0$).

Before giving the proof, we need the following analytical lemma.

\begin{lemma}
\label{lem:estim-transport-euler} For $\omega\in L^{2}$, define $F(\omega):=
(K\ast\omega)\cdot\nabla\omega$; then for any $\alpha\in(0,1)$ it holds
\begin{equation}
\label{eq:estim-transport-euler}\| F(\omega)-F(\tilde{\omega})\|_{H^{-\alpha
-1}} \lesssim_{\alpha}\| \omega-\tilde{\omega}\|_{H^{-\alpha}} (\|
\omega\|_{L^{2}} + \| \tilde\omega\|_{H^{1}})\quad\forall\,\omega\in
L^{2},\tilde{\omega}\in H^{1}.
\end{equation}

\end{lemma}

\begin{proof}
It holds
\[\| F(\omega)-F(\tilde{\omega})\|_{H^{-\alpha-1}} \leq \| [K\ast (\omega-\tilde{\omega})]\cdot\nabla \omega\|_{H^{-\alpha-1}}+ \| (K\ast \tilde{\omega}) \cdot\nabla (\omega-\tilde{\omega})\|_{H^{-\alpha-1}}=:I_1 +I_2.
\]
Applying Lemma \ref{lem-estimate}(d) with $\beta=\alpha$, we can estimate $I_1$ by
\begin{align*}
I_1 \lesssim_\alpha \| K\ast (\omega-\tilde{\omega})\|_{H^{1-\alpha}} \| \omega\|_{L^2} \lesssim \| \omega-\tilde{\omega}\|_{H^{-\alpha}} \| \omega\|_{L^2};
\end{align*}
on the other hand, invoking Lemma \ref{lem-estimate}(b) for $I_2$ provides
\[
I_2 \lesssim_\alpha \| K\ast \tilde{\omega}\|_{H^2} \| \omega-\tilde{\omega}\|_{H^{-\alpha}} \lesssim \| \tilde{\omega}\|_{H^1} \| \omega-\tilde{\omega}\|_{H^{-\alpha}}.
\]
Combining the two estimates gives the conclusion.
\end{proof}

\begin{proof}[Proof of Theorem \ref{thm-1}]
Let $\omega,\tilde{\omega}$ be solutions as above, for the same initial data $\omega_0\in L^2$; let $F$ be defined as in Lemma \ref{lem:estim-transport-euler}. By Section \ref{sec:mild-form}, we know that $\omega,\tilde{\omega}$ both satisfy the mild formulation; by the same reasoning as in Section \ref{subsec-heuristic}, their difference $\xi:=\omega-\tilde{\omega}$ solves
\[
\xi_t = -\int_0^t e^{(\kappa+\nu)(t-s)\Delta} [F(\omega_s)-F(\tilde{\omega}_s)]\, \d s - Z_t,
\]
where the stochastic convolution $Z$ is given by
\[
Z_t = \sqrt{2\kappa}\int_0^t \sum_k \theta_k e^{(\kappa+\nu)(t-s)\Delta} (\sigma_k\cdot\nabla \omega_s)\, \d W^k_s.
\]
By Lemmas \ref{lem:heat-kernel} and \ref{lem:estim-transport-euler}, we can estimate $\xi$ as follows:
\begin{align*}
\| \xi_t\|_{H^{-\alpha}}^2
& \lesssim_\alpha \frac{1}{\kappa+\nu} \int_0^t \| F(\omega_s)-F(\tilde{\omega}_s)\|^2_{H^{-\alpha-1}}\, \d s + \| Z_t\|_{H^{-\alpha}}^2\\
& \lesssim_\alpha \frac1{\kappa+\nu} \int_0^t \| \xi_s\|_{H^{-\alpha}}^2 \big(\| \omega_s \|_{L^2}^2 + \|\tilde \omega_s\|_{H^1}^2 \big)\, \d s + \| Z_t \|_{H^{-\alpha}}^2.
\end{align*}
Gronwall's inequality then implies the existence of $C=C(\alpha)$ such that
\begin{equation}\label{eq:intermediate-estim}
\| \xi_t\|_{H^{-\alpha}}^2 \lesssim \bigg(\sup_{t\in [0,T]} \| Z_t \|_{H^{-\alpha}}^2\bigg) \exp\bigg( \frac{C}{\kappa+\nu} \int_0^T  \big(\| \omega_s \|_{L^2}^2 + \|\tilde \omega_s\|_{H^1}^2\big)\, \d s \bigg) .
\end{equation}
Recalling that $\omega$ and $\tilde{\omega}$ satisfy respectively the a priori estimates \eqref{L-2-bound} and \eqref{L-2-bound-determ}, we arrive at
$$\| \xi_t\|_{H^{-\alpha}}^2 \lesssim \bigg(\sup_{t\in [0,T]} \| Z_t \|_{H^{-\alpha}}^2\bigg) \exp\bigg( C\, \frac{1+T(\kappa+\nu)}{(\kappa+\nu)^2} \|\omega_0 \|_{L^2}^2 \bigg) .$$
Taking expectation on both sides and applying \eqref{eq:estim-Z-interp} with $\delta= \kappa+\nu \geq\kappa$ yield the assertion (i).

If $\nu>0$, we can employ the a priori estimates in a different manner, giving
$$\aligned
\int_0^T \big(\| \omega_s\|_{L^2}^2 + \| \tilde{\omega}_s\|_{H^1}^2 \big)\, \d s
&\lesssim \int_0^{+\infty} \big(\| \nabla\omega_s\|_{L^2}^2 + \| \nabla\tilde{\omega}_s\|_{L^2}^2 \big)\, \d s \\
&\lesssim \bigg(\frac{1}{\nu} + \frac{1}{\kappa+\nu} \bigg)\|\omega_0 \|_{L^2}^2  \leq \frac{2}{\nu} \|\omega_0 \|_{L^2}^2;
\endaligned $$
inserting this estimate in \eqref{eq:intermediate-estim} and taking expectation as before readily gives (ii).
\end{proof}

Let us stress the importance of the asymmetric estimate
\eqref{eq:estim-transport-euler} in our analysis, especially in order to
achieve a convergence rate which is uniform in $\nu\geq0$. Indeed we exploit
crucially the information on the regularity of $\tilde{\omega}$, which is
better than the one available for $\omega$ (for $\nu=0$ estimate
\eqref{L-2-bound} only gives a control on its $L^{2}$-norm). The same idea
will be used in the next sections for other fluid dynamics equations.

We complete this section with

\begin{proof}[Proof of Corollary \ref{cor-approx-uniq}]
The proof is very simple. Recall that $\tilde\omega$ is the unique solution to the deterministic 2D Navier-Stokes equation \eqref{determ-NS} with initial data $\omega_0\in L^2(\T^2)$; we regard its law as a delta Dirac mass $\delta_{\tilde\omega}$ on $C([0,T], H^{-\alpha})$. Then for any $Q,Q'\in \mathcal L_\theta$, by the triangle inequality for the Wasserstein distance,
$$d_p(Q,Q') \leq d_p(Q,\delta_{\tilde\omega}) + d_p(Q',\delta_{\tilde\omega}).$$
Let $\omega$ (resp. $\omega'$) be a weak solution to the stochastic 2D Euler equation \eqref{stoch-NS} (taking $\nu=0$) with law $Q$ (resp. $Q'$); here $\omega$ and $\omega'$ might be defined on two different probability spaces, but we do not distinguish the expectations below. Then we have
$$d_p(Q,Q') \leq \E\Big[\|\omega- \tilde\omega\|_{C([0,T], H^{-\alpha})}^p \Big]^{1/p} + \E\Big[\|\omega'- \tilde\omega\|_{C([0,T], H^{-\alpha})}^p \Big]^{1/p}.$$
Combining this inequality with Theorem \ref{thm-1} and choosing $\eps=\alpha/2$, we immediately obtain the desired result. The second inequality follows from the first one and Example \ref{ex:coefficients}-(1).
\end{proof}

\subsection{2D Boussinesq system}

\label{sec:boussinesq}

The 2D Boussinesq system models the evolution of velocity field of an
incompressible fluid under a vertical force, which is proportional to some
scalar field such as the temperature, the latter being transported by the
former. We refer to \cite{Majda} for the geophysical background of the system.
In this section we aim at deriving similar quantitative estimates between the
solutions to the stochastic 2D inviscid Boussinesq model (in vorticity form)
\begin{equation}
\label{eq:stoch-boussinesq}%
\begin{cases}
\mathrm{d} \gamma^{1} + u^{1}\cdot\nabla\gamma^{1} \mathrm{d} t +
\circ\mathrm{d} W\cdot\nabla\gamma^{1} = \nu\Delta\gamma^{1} \mathrm{d} t,\\
\mathrm{d} \omega^{1} + u^{1}\cdot\nabla\omega^{1} \mathrm{d} t +
\circ\mathrm{d} W\cdot\nabla\omega^{1} = \partial_{1} \gamma^{1} \mathrm{d} t
\end{cases}
\end{equation}
and those of the deterministic viscous system
\begin{equation}
\label{eq:determ-boussinesq}%
\begin{cases}
\partial_{t} \gamma^{2} + u^{2}\cdot\nabla\gamma^{2} = (\kappa+\nu)
\Delta\gamma^{2},\\
\partial_{t} \omega^{2} + u^{2}\cdot\nabla\omega^{2} = \kappa\Delta\omega^{2}
+ \partial_{1} \gamma^{2}.
\end{cases}
\end{equation}
In the above equations, $u^{1}= K\ast\omega^{1}$ and $u^{2}= K\ast\omega^{2}$
where $K$ is still the Biot-Savart kernel. As before, we take identical
initial data $\omega^{1}_{0}=\omega^{2}_{0}=\omega_{0}\in L^{2}(\mathbb{T}%
^{2})$, $\gamma^{1}_{0}=\gamma^{2}_{0}=\gamma_{0}\in L^{2}(\mathbb{T}^{2})$,
and the noise $W$ is the same as in Section \ref{subsec-quantitat}. Recall
that there exist weak solutions to \eqref{eq:stoch-boussinesq} satisfying the
following a priori estimates: $\mathbb{P}$-a.s.,
\begin{equation}
\label{eq:energy-bounds-boussinesq}\sup_{t\in[0,T]} \| \gamma^{1}_{t}%
\|^{2}_{L^{2}} + \nu\int_{0}^{T} \| \gamma^{1}_{t} \|_{H^{1}}\, \mathrm{d} t
\leq\|\gamma_{0}\|_{L^{2}}^{2}, \quad\sup_{t\in[0,T]} \| \omega^{1}_{t}%
\|^{2}_{L^{2}} \leq C_{\nu,T} \big(\|\omega_{0}\|^{2}_{L^{2}}+\|\gamma
_{0}\|^{2}_{L^{2}} \big)
\end{equation}
for some deterministic constant $C_{\nu,T}>0$, see \cite[Theorem 2.2]{Luo};
moreover
\[
\aligned
&  \sup_{t\in[0,T]} \| \gamma^{2}_{t}\|^{2}_{L^{2}} + (\kappa+\nu) \int%
_{0}^{T} \| \gamma^{2}_{t} \|_{H^{1}}^{2}\, \mathrm{d} t \leq\|\gamma
_{0}\|_{L^{2}}^{2}, \\
&  \sup_{t\in[0,T]} \| \omega^{2}_{t}%
\|^{2}_{L^{2}} + \kappa\int_{0}^{T} \| \omega^{2}_{t} \|_{H^{1}}^{2}\,
\mathrm{d} t \leq C_{\nu,T} \big(\|\omega_{0}\|^{2}_{L^{2}} +\|\gamma
_{0}\|^{2}_{L^{2}} \big)
\endaligned
\]
uniformly in $\kappa\geq0$, for the same constant $C_{\nu,T}$ (indeed the
presence of the additional viscosity $\kappa\Delta$ can only further improve
the control on the energy).

As before, we define two martingale terms $M,N$ by setting
\[
M_{t} = \int_{0}^{t} \nabla\gamma^{1}_{s}\cdot\mathrm{d} W_{s}, \quad N_{t} =
\int_{0}^{t} \nabla\omega^{1}_{s}\cdot\mathrm{d} W_{s}
\]
as well as the associated stochastic convolutions
\[
Z_{t} = \int_{0}^{t} e^{(\kappa+\nu)(t-s)\Delta}\, \mathrm{d} M_{s}%
,\quad\tilde{Z}_{t} = \int_{0}^{t} e^{\kappa(t-s)\Delta}\, \mathrm{d} N_{s}.
\]
Passing to It\^o form of the system \eqref{eq:stoch-boussinesq}, and rewriting
it and \eqref{eq:determ-boussinesq} in the corresponding mild formulations, we
arrive at
\begin{equation}
\label{eq:mild-form-boussinesq}%
\begin{cases}
\gamma^{1}_{t} = e^{(\kappa+\nu)t \Delta}\gamma_{0} -\int_{0}^{t}
e^{(\kappa+\nu)(t-s)\Delta} (u^{1}_{s}\cdot\nabla\gamma^{1}_{s})\,\mathrm{d} s
- Z_{t},\\
\omega^{1}_{t} = e^{\kappa t \Delta} \omega_{0}-\int_{0}^{t} e^{\kappa
(t-s)\Delta} (u^{1}_{s}\cdot\nabla\omega^{1}_{s} -\partial_{1} \gamma^{1}%
_{s})\, \mathrm{d} s - \tilde{Z}_{t},\\
\gamma^{2}_{t} = e^{(\kappa+\nu)t \Delta}\gamma_{0} -\int_{0}^{t}
e^{(\kappa+\nu)(t-s)\Delta} (u^{2}_{s}\cdot\nabla\gamma^{2}_{s})\, \mathrm{d}
s,\\
\omega^{2}_{t} = e^{\kappa t \Delta} \omega_{0}-\int_{0}^{t} e^{\kappa
(t-s)\Delta} (u^{2}_{s}\cdot\nabla\omega^{2}_{s} -\partial_{1} \gamma^{2}%
_{s})\, \mathrm{d} s.
\end{cases}
\end{equation}
Setting $\lambda=\gamma^{1}-\gamma^{2},\, \xi=\omega^{1}-\omega^{2}$, the
differences satisfy the equations
\begin{equation}
\label{eq:diff-sol-boussinesq}%
\begin{cases}
\lambda_{t} = - \int_{0}^{t} e^{(\kappa+\nu)(t-s)\Delta} (u^{1}_{s}\cdot
\nabla\gamma^{1}_{s}-u^{2}_{s}\cdot\nabla\gamma_{s}^{2})\,\mathrm{d} s -
Z_{t}\\
\xi_{t} = -\int_{0}^{t} e^{\kappa(t-s)\Delta} (u^{1}_{s}\cdot\nabla\omega
^{1}_{s}-u^{2}_{s}\cdot\nabla\omega^{2}_{s} -\partial_{1}\lambda_{s})\,
\mathrm{d} s - \tilde{Z}_{t}.
\end{cases}
\end{equation}
With these preparations, we are ready to give an estimate for $(\lambda,\xi)$.

\begin{lemma}
\label{lem:gronwall-bound-boussinesq} Under the above assumptions, for any
$\alpha\in(0,1)$, there exists a deterministic $C=C(\alpha, \nu, T)>0$ such
that $\mathbb{P}$-a.s. it holds
\begin{equation}%
\begin{split}
\label{eq:gronwall-bound-boussinesq}\sup_{t\in[0,T]} \big( \| \lambda
_{t}\|_{H^{-\alpha}}+\| \xi_{t}\|_{H^{-\alpha}}\big)  &  \lesssim
\exp\bigg[\frac{C}{\kappa} \bigg( 1+ \| \gamma_{0}\|_{L^{2}}^{2}+\| \omega
_{0}\|_{L^{2}}^{2} + \frac{\| \omega_{0}\|_{L^{2}}^{2}}{\kappa}
\bigg) \bigg]\\
&  \quad\times\bigg(\sup_{t\in[0,T]}\| Z_{t}\|_{H^{-\alpha}}+\sup_{t\in
[0,T]}\| \tilde Z_{t}\|_{H^{-\alpha}}\bigg).
\end{split}
\end{equation}

\end{lemma}

\begin{proof}
Reasoning as before, by Lemma \ref{lem:heat-kernel} we have
\begin{equation}\label{lem:gronwall-bound-boussinesq.1}
\aligned
\| \lambda_t\|_{H^{-\alpha}}^2
& \lesssim \frac1{\kappa+\nu} \int_0^t \big\|u^1_s\cdot\nabla\gamma^1_s-u^2_s\cdot \nabla\gamma_s^2 \big\|_{H^{-\alpha-1}}^2\, \d s + \|Z_t\|_{H^{-\alpha}}^2, \\
\|\xi_t \|_{H^{-\alpha}}^2
& \lesssim \frac1\kappa \int_0^t \big\|u^1_s\cdot\nabla\omega^1_s-u^2_s\cdot\nabla\omega^2_s -\partial_1\lambda_s \big\|_{H^{-\alpha-1}}^2\, \d s + \|\tilde Z_t\|_{H^{-\alpha}}^2 .
\endaligned
\end{equation}
For $s\in [0,T]$, define
$$I_1(s) = \big\|u^1_s\cdot\nabla\gamma^1_s-u^2_s\cdot \nabla\gamma_s^2 \big\|_{H^{-\alpha-1}}^2, \quad I_2(s)= \big\|u^1_s\cdot\nabla\omega^1_s-u^2_s\cdot\nabla\omega^2_s -\partial_1\lambda_s \big\|_{H^{-\alpha-1}}^2. $$
Note that $u^1_s-u^2_s = K\ast \xi_s$ and $\gamma^1_s- \gamma_s^2 = \lambda_s$; arguing as in the proof of Lemma \ref{lem:estim-transport-euler}, we can estimate $I_1$ by applying respectively points (d) and (b) of Lemma \ref{lem-estimate} as follows:
\begin{align*}
I_1 (s)
& \lesssim \| (K\ast \xi_s)\cdot\nabla \gamma^1_s\|_{H^{-\alpha-1}}^2 + \| (K\ast \omega^2_s)\cdot\nabla\lambda_s\|_{H^{-\alpha-1}}^2\\
& \lesssim_\alpha \| K\ast \xi_s\|_{H^{1-\alpha}}^2 \|\gamma^1_s\|_{L^2}^2 + \| K\ast\omega^2_s\|_{H^2}^2 \| \lambda_s\|_{H^{-\alpha}}^2\\
& \lesssim \|\gamma^1_s\|_{L^2}^2 \| \xi_s\|_{H^{-\alpha}}^2  + \| \omega^2_s\|_{H^1}^2 \| \lambda_s\|_{H^{-\alpha}}^2.
\end{align*}
For the term $I_2$ we can apply directly Lemma \ref{lem:estim-transport-euler}:
\begin{align*}
I_2(s)
&\lesssim \| u^1_s\cdot\nabla\omega^1_s-u^2_s\cdot\omega^2_s\|_{H^{-\alpha-1}}^2 + \| \partial_1\lambda_s\|_{H^{-\alpha-1}}^2
\lesssim_\alpha \big(\| \omega^1_s\|_{L^2}^2 + \| \omega^2_s\|_{H^1}^2 \big) \|\xi_s\|_{H^{-\alpha}}^2 + \| \lambda_s\|_{H^{-\alpha}}^2.
\end{align*}
Substituting the above estimates into \eqref{lem:gronwall-bound-boussinesq.1}, using $(\kappa+\nu)^{-1}\leq \kappa^{-1}$, we arrive at
\begin{align*}
\| \lambda_t\|_{H^{-\alpha}}^2+\| \xi_t\|_{H^{-\alpha}}^2
\lesssim_\alpha &\, \frac1\kappa \int_0^t \big(1+\|\gamma^1_s\|_{L^2}^2 + \| \omega^1_s\|_{L^2}^2 + \|\omega^2_s\|_{H^1}^2\big) \big(\| \lambda_s\|_{H^{-\alpha}}^2+\| \xi_s\|_{H^{-\alpha}}^2\big )\, \d s\\
& + \big(\|Z_t\|_{H^{-\alpha}}^2 + \|\tilde Z_t\|_{H^{-\alpha}}^2 \big).
\end{align*}
The a priori estimate \eqref{eq:energy-bounds-boussinesq} gives us
\begin{align*}
\| \lambda_t\|_{H^{-\alpha}}^2+\| \xi_t\|_{H^{-\alpha}}^2
\lesssim_\alpha & \, \frac1\kappa \int_0^t \big[ \tilde C_{\nu,T} \big(1+ \|\gamma_0\|_{L^2}^2 + \| \omega_0\|_{L^2}^2 \big) + \|\omega^2_s\|_{H^1}^2\big] \big(\| \lambda_s\|_{H^{-\alpha}}^2+\| \xi_s\|_{H^{-\alpha}}^2\big )\, \d s\\
&+ \big(\|Z_t\|_{H^{-\alpha}}^2 + \|\tilde Z_t\|_{H^{-\alpha}}^2 \big),
\end{align*}
where $\tilde C_{\nu,T}=1+ C_{\nu,T} $. Finally, applying Gronwall's lemma we obtain the conclusion.
\end{proof}

Combining the above result with the maximal estimate in Corollary
\ref{cor-stoch-convol} for stochastic convolution, we immediately get

\begin{proposition}
\label{prop-boussinesq} Assume $\|\theta\|_{\ell^{2}} =1$. For any $\alpha
\in(0,1)$ and $\varepsilon\in(0,\alpha)$, we have
\[%
\begin{split}
\mathbb{E}\bigg[ \sup_{t\in[0,T]} \Big( \| \gamma^{1}_{t}- \gamma^{2}%
_{t}\|_{H^{-\alpha}}+\| \omega^{1}_{t}-\omega^{2}_{t}\|_{H^{-\alpha}%
}\Big) \bigg] \lesssim_{\varepsilon,T}  &  \, \exp\bigg[\frac{C}{\kappa}
\bigg( 1+ \| \gamma_{0}\|_{L^{2}}^{2}+\| \omega_{0}\|_{L^{2}}^{2} + \frac{\|
\omega_{0}\|_{L^{2}}^{2}}{\kappa} \bigg) \bigg]\\
&  \, \times\kappa^{\varepsilon/2} \|\theta\|_{\ell^{\infty}}^{\alpha-
\varepsilon} \big( \|\gamma_{0}\|_{L^{2}} + \|\omega_{0}\|_{L^{2}} \big).
\end{split}
\]

\end{proposition}

\subsection{mSQG equations}

\label{sec:mSQG} The mSQG (modified Surface Quasi-Geostrophic) equation is an
interpolation between the vorticity form of 2D Euler equation and the SQG
equation, the latter being widely used in meteorological and oceanic flows to
describe the temperature in a rapidly rotating stratified fluid with uniform
potential vorticity (cf. \cite{HPGS}).

For $\beta\in(0,1)$, set $K_{\beta}:= \nabla^{\perp}\cdot(-\Delta
)^{-\frac{1+\beta}{2}}$; note that $K_{1}$ is the Biot-Savart kernel while
$K_{0}$ is the kernel in the SQG equation. We see that the regularizing effect
of $K_{\beta}$ is increasing in $\beta$. The aim of this section is to obtain
rates of convergence for the stochastic mSQG equation
\[
\mathrm{d} \omega+ (K_{\beta}\ast\omega)\cdot\nabla\omega\, \mathrm{d} t +
\circ\mathrm{d} W\cdot\nabla\omega=0
\]
to its deterministic viscous counterpart
\[
\partial_{t} \tilde\omega+ (K_{\beta}\ast\tilde\omega)\cdot\nabla\tilde\omega=
\kappa\Delta\tilde\omega.
\]
Assuming that $\omega_{0}=\tilde\omega_{0}\in L^{2}$, we have the a priori
estimates
\begin{equation}
\label{eq:energy-bounds-mSQG}\sup_{t\in[0,T]} \|\omega_{t}\|_{L^{2}} \leq\|
\omega_{0}\|_{L^{2}} \quad(\mathbb{P}\text{-a.s.}),\quad\sup_{t\in[0,T]}
\bigg\{ \|\tilde\omega_{t} \|_{L^{2}}^{2} + 2\kappa\int_{0}^{t} \|
\nabla\tilde\omega_{s}\|_{L^{2}}^{2} \, \mathrm{d} s \bigg\}\leq\| \omega
_{0}\|_{L^{2}}^{2};
\end{equation}
see respectively Theorem 2.1 and Theorem 4.1 from \cite{LuoSaal}.

As before, writing both equations in mild formulation (after passing to It\^o
form), defining the martingale $M$ and associated stochastic convolution
\[
Z_{t} = \int_{0}^{t} e^{\kappa(t-s)\Delta} \,\mathrm{d} M_{s}= \sqrt{2\kappa}
\int_{0}^{t} \sum_{k} \theta_{k} e^{\kappa(t-s)\Delta} (\sigma_{k} \cdot
\nabla\omega_{s} ) \,\mathrm{d} W^{k}_{s},
\]
we arrive at an equation for the difference $\xi=\omega-\tilde\omega$ of the
form
\begin{equation}
\label{eq:diff-sol-mSQG}\xi_{t} = - \int_{0}^{t} e^{\kappa(t-s)\Delta}
\big[(K_{\beta}\ast\omega_{s})\cdot\nabla\omega_{s}- (K_{\beta}\ast
\tilde\omega_{s}) \cdot\nabla\tilde\omega_{s} \big]\, \mathrm{d} s - Z_{t}.
\end{equation}

Before going into calculations, let us make the following remark. As the
kernel $K_{\beta}$ is not as regularizing as the classical Biot-Savart kernel
$K=K_{1}$, we are not able to prove an estimate of the form
\eqref{eq:estim-transport-euler}; consequently, the strategy employed in
Sections \ref{subsec-proof}-\ref{sec:boussinesq} does not trivially extend to
mSQG. The challenge here is entirely analytic, as the bounds for the
stochastic convolution $Z$ are the same as in previous sections; we must adopt
slightly different estimates.

\begin{proposition}
\label{prop:mSQG} Fix $\delta>0$, $\beta\in(0,1)$, $q>2/\beta$ and $\alpha
\in(0,\beta)$. Then for any $T<\infty$ there exists a constant $C=C(T,\delta
,\alpha,\beta,q)$ such that for any $\kappa\geq\delta$ and any two solutions
$\omega,\tilde{\omega}$ as above it holds
\begin{equation}
\label{eq:gronwall-mSQG-bis}\sup_{t\in[0,T]} \| \omega_{t}-\tilde{\omega}%
_{t}\|_{H^{-\alpha}} \leq C \exp\big(C \| \omega_{0}\|_{L^{2}}^{q}
\big) \sup_{t\in[0,T]}\| Z_{t}\|_{H^{-\alpha}}.
\end{equation}

\end{proposition}

\begin{proof}
It holds
\begin{equation*}\aligned
\xi_t &= - \int_0^t e^{\kappa(t-s)\Delta} [(K_\beta\ast\xi_s)\cdot\nabla\omega_s]\, \d s -\int_0^t e^{\kappa(t-s)\Delta} [(K_\beta\ast\tilde\omega_s)\cdot\nabla\xi_s]\, \d s - Z_t \\
&=: I^1_t + I^2_t - Z_t.
\endaligned
\end{equation*}
Using Lemma \ref{lem:heat-kernel-estim} and Lemma \ref{lem-estimate}(d), we can estimate the first term as follows:
\begin{align*}
\| I^1_t\|_{H^{-\alpha}}
& \leq \int_0^t \big\| e^{\kappa(t-s)\Delta} [(K_\beta\ast\xi_s)\cdot\nabla\omega_s] \big\|_{H^{-\alpha}}\, \d s\\
& \lesssim \kappa^{-1 + \beta/2} \int_0^t |t-s|^{-1 + \beta/2}\, \| (K_\beta\ast\xi_s)\cdot\nabla \omega_s\|_{H^{\beta-\alpha-2}}\, \d s \\
& \lesssim \kappa^{-1 + \beta/2} \int_0^t |t-s|^{-1 + \beta/2}  \| K_\beta \ast\xi_s \|_{H^{\beta-\alpha}} \| \omega_s\|_{L^2}\, \d s\\
& \lesssim \kappa^{-1 + \beta/2}\, \|\omega_0\|_{L^2} \int_0^t |t-s|^{-1+\beta/2}\, \| \xi_s\|_{H^{-\alpha}}\, \d s,
\end{align*}
where the last step follows from the first bound in \eqref{eq:energy-bounds-mSQG} and the regularizing properties of $K_\beta$. By H\"older's inequality ($q'$ is the conjugate number of $q$),
\begin{align*}
\| I^1_t\|_{H^{-\alpha}}
& \lesssim \kappa^{-1 + \beta/2}\, \|\omega_0\|_{L^2} \bigg(\int_0^t |t-s|^{(-1+\beta/2)q'} \d s\bigg)^{1/q'} \bigg(\int_0^t \| \xi_s\|_{H^{-\alpha}}^q \,\d s \bigg)^{1/q}\\
& \lesssim_T \kappa^{-1 + \beta/2}\, \|\omega_0\|_{L^2} \bigg(\int_0^t \| \xi_s\|_{H^{-\alpha}}^q\,\d s \bigg)^{1/q}
\end{align*}
where by the assumption $q>2/\beta$ the integral in the first line is finite.

For the second term we use Lemma \ref{lem:heat-kernel}, together with Lemma \ref{lem-estimate}(b) and the hypothesis $\beta>\alpha$ to obtain
\begin{align*}
\| I^2_t\|_{H^{-\alpha}}
& \lesssim \kappa^{-1/2} \bigg(\int_0^t \|(K_\beta\ast\tilde\omega_s)\cdot \nabla\xi_s \|_{H^{-\alpha-1}}^2\, \d s\bigg)^{1/2}\\
& \lesssim \kappa^{-1/2} \bigg(\int_0^t \| K_\beta\ast\tilde{\omega}_s\|_{H^{1+\beta}}^2\, \| \xi_s\|_{H^{-\alpha}}^2\, \d s\bigg)^{1/2}\\
& \lesssim \kappa^{-1/2} \bigg(\int_0^t \| \tilde{\omega}_s\|_{H^1}^2\, \| \xi_s\|_{H^{-\alpha}}^2\, \d s\bigg)^{1/2}.
\end{align*}
For $q>2$, by H\"older's inequality,
\begin{align*}
\| I^2_t\|_{H^{-\alpha}}
& \lesssim \kappa^{-1/2} \bigg( \int_0^t \| \tilde{\omega}_s\|_{H^1}^2 \, \d s\bigg)^{\frac{q-2}{2q}}\, \bigg( \int_0^t \| \tilde{\omega}_s \|_{H^1}^2 \| \xi_s\|_{H^{-\alpha}}^q\, \d s\bigg)^{1/q}\\
& \lesssim \kappa^{-1+ \frac1{q}} \| \omega_0\|_{L^2}^{1-\frac{2}{q}} \bigg( \int_0^t \| \tilde{\omega}_s\|_{H^1}^2 \| \xi_s\|_{H^{-\alpha}}^q\, \d s\bigg)^{1/q}.
\end{align*}
Combining the above estimates we obtain
\begin{align*}
\| \xi_t\|_{H^{-\alpha}}^q
& \lesssim \| I^1_t\|_{H^{-\alpha}}^q + \| I^2_t \|_{H^{-\alpha}}^q + \|Z_t\|_{H^{-\alpha}}^q\\
& \lesssim \int_0^t \Big( \kappa^{-q + q\beta/2} \|\omega_0\|_{L^2}^q + \kappa^{-q+1} \|\omega_0\|_{L^2}^{q-2} \| \tilde{\omega}_s\|_{H^1}^2 \Big) \| \xi_s\|_{H^{-\alpha}}^q\, \d s + \| Z_t\|_{H^{-\alpha}}^q.
\end{align*}
By Gronwall's lemma and the second bound in \eqref{eq:energy-bounds-mSQG}, using the assumption $\kappa\geq \delta$, we find $C=C(T,\delta,\alpha,\beta,q)>0$ such that
\begin{equation*}
\sup_{t\in [0,T]} \| \xi_t\|_{H^{-\alpha}}^q
\lesssim \exp\Big[ C(1+T)\| \omega_0\|_{L^2}^q \Big] \sup_{t\in [0,T]} \| Z_t\|_{H^{-\alpha}}^q;
\end{equation*}
up to relabelling $C$, the conclusion follows.
\end{proof}

Applying the maximal estimate for stochastic convolution, we obtain

\begin{corollary}
Consider parameters $\delta, \beta, q,\alpha,T$ as above, $C$ be the constant
from Proposition \ref{prop:mSQG}; then for any $p\in[1,\infty)$, $\kappa
\geq\delta$ and any $\varepsilon\in(0,\alpha]$, we have
\[
\mathbb{E}\bigg[\sup_{t\in[0,T]} \| \omega_{t}-\tilde{\omega}_{t}%
\|_{H^{-\alpha}}^{p} \bigg]^{1/p} \lesssim_{\varepsilon, p,T} \kappa
^{\varepsilon/2} \|\theta\|_{\ell^{\infty}}^{\alpha- \varepsilon}\, C
\|\omega_{0}\|_{L^{2}} \exp\big(C \| \omega_{0}\|_{L^{2}}^{q} \big).
\]

\end{corollary}

\section{Blow-up probability estimates}\label{sec:blow-up}

The purpose of this section is to prove Theorem \ref{thm-3}. To this end, we
first make some necessary preparations in Section \ref{subsec-KS-preparations}%
, and then provide the proof in Section \ref{subsec-KS-proof}, following the
main idea in the previous sections.

\subsection{Preliminaries on the Keller-Segel system}

\label{subsec-KS-preparations}

Let us start by reformulating system \eqref{stoch-keller-segel} in a way which
is more suited for our purposes. For any $f\in L^{2}(\mathbb{T}^{2})$ we
define the operator $\nabla^{-1} f= \nabla(-\Delta)^{-1} (f-f_{\mathbb{T}^{2}%
})$; $\nabla^{-1}$ extends to a continuous linear operator from $H^{s}$ to
$H^{s+1}$ for any $s\in\mathbb{R}$ and satisfies $\nabla\cdot\nabla^{-1} f =
-f +f_{\mathbb{T}^{2}}$, $(\nabla^{-1} f)_{\mathbb{T}^{2}}=0$ for regular $f$.
With this notation, system \eqref{stoch-keller-segel} can be written in a more
compact form:
\[
\mathrm{d} \rho= \big(\Delta\rho- \nabla\cdot[\rho\nabla^{-1}\rho]
\big)\,\mathrm{d} t + \circ\mathrm{d} W\cdot\nabla\rho.
\]
Observe that if $\rho$ satisfies the SPDE, then it has constant mean
$\rho_{\mathbb{T}^{2}}(t) = \rho_{\mathbb{T}^{2}}(0) =: \bar\rho>0$, since $W$
is spatially divergence free. Defining $u=\rho-\bar{\rho}$ and using the
properties of $\nabla^{-1}$, we obtain
\begin{equation}
\label{eq:stoch-KSE}\mathrm{d} u = \big(\Delta u -\nabla\cdot[u\nabla^{-1}
u]+\bar\rho u \big)\,\mathrm{d} t + \circ\mathrm{d} W\cdot\nabla u.
\end{equation}
Finally, this equation has the following equivalent It\^o form
\begin{equation}
\label{eq:stoch-keller-segel-new}\mathrm{d} u = \big((1+\kappa)\Delta u
-\nabla\cdot[u\nabla^{-1} u]+\bar\rho u \big)\,\mathrm{d} t + \mathrm{d}
W\cdot\nabla u.
\end{equation}
Similarly, if we start from the deterministic system \eqref{keller-segel} with
$\chi=1$ and $(1+\kappa)\Delta$ in place of $\Delta$, we would have
$\rho=\tilde u+\bar{\rho}$ with
\begin{equation}
\label{eq:keller-segel-enhanced-new}\partial_{t} \tilde u = (1+\kappa
)\Delta\tilde u + \bar{\rho} \tilde u - \nabla\cdot[\tilde u\nabla^{-1} \tilde
u].
\end{equation}
The advantage in dealing with $u$ in place of $\rho$ lies in the fact that
$\rho$ blows up if and only if $u$ does, but $u_{\mathbb{T}^{2}}=0$, allowing
easy use of Poincar\'e inequality. However, keep in mind that $\rho_{0}$
encodes the pair $(\bar{\rho},u)$ of data of the problem; also observe that
$\|u_{0}\|_{L^{2}}^{2} + \bar{\rho}^{2} = \|\rho_{0}\|_{L^{2}}^{2}$.

Let us quickly explain the main idea involving estimates on blow-up: we expect
equation \eqref{eq:stoch-keller-segel-new} to be close to
\eqref{eq:keller-segel-enhanced-new} in the scaling limit, at least in some
weak norm $H^{-\alpha}$. Therefore, blow-up can be delayed if we can show
that: i) the solution $\tilde{u}$ to \eqref{eq:keller-segel-enhanced-new}
exists globally; ii) blow-up for \eqref{eq:stoch-keller-segel-new} in strong
norms only takes place if $\| u\|_{H^{-\alpha}}$ blows up.

Below we verify that both requirements are met.

\begin{lemma}
\label{lem:basin-KS-new} There exists $C>0$ with the following property: given
$\rho_{0}\in L^{2}$, if $\kappa\geq C \|\rho_{0}\|_{L^{2}}^{2} + 1$, then
global existence holds for \eqref{eq:keller-segel-enhanced-new} and moreover
the solution satisfies
\begin{equation}
\label{eq:basin-KS-new}\max\bigg\{ \sup_{t\geq0} \| \rho_{t}\|_{L^{2}}^{2},
\int_{0}^{+\infty} \|\nabla\rho_{t}\|_{L^{2}}^{2}\, \mathrm{d} t\bigg\} \leq
\|\rho_{0}\|_{L^{2}}^{2}.
\end{equation}

\end{lemma}

\begin{proof}
The energy balance for \eqref{eq:keller-segel-enhanced-new} can be computed as follows:
\begin{align*}
\frac{\d}{\d t} \|u\|_{L^2}^2 + 2(1+\kappa)\|\nabla u\|_{L^2}^2
= 2\bar{\rho}\, \|u\|_{L^2}^2 + 2 \langle u\nabla u, \nabla^{-1}u \rangle
= 2\bar{\rho}\, \|u\|_{L^2}^2 + \|u\|_{L^3}^3.
\end{align*}
By Sobolev embedding, interpolation and Young's inequality we have
\[
\|u\|_{L^3}^3 \lesssim \|u\|_{H^{1/3}}^3 \lesssim \| u\|_{L^2}^2 \|\nabla u\|_{L^2} \leq \|\nabla u\|_{L^2}^2 + c \|u\|_{L^2}^4
\]
for some constant $c>0$. By the Poincar\'e inequality $\|\nabla u\|_{L^2}^2 \geq 4\pi^2 \|u\|_{L^2}^2$ we deduce that
\begin{equation}\label{KS-energy-balance}
\frac{\d}{\d t} \|u\|_{L^2}^2 + \|\nabla u\|_{L^2}^2 \leq -\big(8\pi^2 \kappa - 2\bar{\rho} - c \|u\|_{L^2}^2 \big) \|u\|_{L^2}^2.
\end{equation}
We claim that the constant $C$ in the statement can be chosen as
\[
C = \frac{c+1}{8\pi^2}
\]
where $c$ is the constant appearing above. Indeed, if $\kappa\geq C \|\rho_0\|_{L^2}^2+1$, then
\begin{align*}
8\pi^2 \kappa - 2 \bar{\rho}-c\| u_0\|_{L^2}^2
\geq 8\pi^2 \kappa - 2 \|\rho_0\|_{L^2} - c \| \rho_0\|_{L^2}^2
\geq 8\pi^2 \kappa -  (c+1) \| \rho_0\|_{L^2}^2 - 1 \geq 1.
\end{align*}
This implies that $\frac{\d}{\d t} \|u_t\|_{L^2}^2 < 0$ at the initial time $t=0$, so the energy is decreasing, enforcing the fact that $8 \pi^2 \kappa - 2 \bar{\rho}- c \|u\|_{L^2}^2 \geq 1$ will also be true at subsequent times and so that
\[
\frac{\d }{\d t} \| u\|_{L^2}^2 \leq - \| u\|_{L^2}^2 \quad \forall\, t\geq 0.
\]
As a consequence $\| u_t\|_{L^2}^2 \leq e^{-t} \| u_0\|_{L^2}^2$, which together with the energy balance \eqref{KS-energy-balance} also implies
\[
\int_0^{+\infty} \|\nabla u_t\|_{L^2}^2\, \d t \leq \|u_0\|_{L^2}^2.
\]
The conclusion readily follows from the relations $\nabla u_t = \nabla \rho_t$ and $\|\rho_t\|_{L^2}^2= \| u_t\|_{L^2}^2 + \bar{\rho}^2$.
\end{proof}

Given $\alpha>0$ to be chosen later, in order to show that $u$ blows up only
if $\|u\|_{H^{-\alpha}}$ does so, we turn to study the following modified
version of \eqref{eq:stoch-KSE}:
\begin{equation}
\label{eq:keller-segel-cutoff-new}\mathrm{d} u = \big\{ \Delta u + \bar{\rho
}\, u - g_{\alpha,R}(u) \nabla\cdot[u\nabla^{-1} u] \big\}\, \mathrm{d} t +
\circ\mathrm{d} W\cdot\nabla u.
\end{equation}
Here $g_{\alpha,R}(u):= g_{R}(\| u\|_{H^{-\alpha}})$ is a cutoff function,
where $g_{R}\in C([0,+\infty);[0,1])$ satisfies $g_{R}\equiv1$ on $[0,R]$,
$g_{R}\equiv0$ on $[R+1,+\infty)$ and is Lipschitz with constant $1$.

\begin{lemma}
\label{lem:keller-segel-cutoff-new} Let $\alpha\in(0,1)$ and $R>0$ be fixed.
Then global existence of solutions holds for
\eqref{eq:keller-segel-cutoff-new} for any initial data $\rho_{0}\in L^{2}$.
Furthermore, there exists a constant $C_{\alpha}$ such that the unique
solution $u$ satisfies
\begin{equation}
\label{eq:estim-keller-segel-cutoff-new}\sup_{t\in[0,T]} \| u_{t}\|_{L^{2}%
}^{2} \leq e^{2\bar{\rho}\, T} \bigg(\| u_{0}\|_{L^{2}}^{2} + \frac{C_{\alpha
}}{\bar\rho}\, (R+1)^{\frac{4}{1-\alpha}}\bigg)\quad\mathbb{P}\mbox{-a.s.}
\end{equation}

\end{lemma}

\begin{proof}
Global existence and uniqueness of solutions follows from \cite[Proposition 3.6]{FGL2}, so we only need to focus on the proof of \eqref{eq:estim-keller-segel-cutoff-new}. Due to the divergence free property and Stratonovich structure of the noise, the energy balance is given by
\[
\frac{\d}{\d t} \|u\|_{L^2}^2 + 2\| \nabla u\|_{L^2}^2 = 2\bar{\rho}\, \|u\|_{L^2}^2 + g_{\alpha,R}(u) \| u\|_{L^3}^3.
\]
As before, we can estimate the last term by Sobolev embedding and interpolation, only replacing $\|\cdot\|_{L^2}$ with $\| \cdot\|_{H^{-\alpha}}$:
\begin{align*}
g_{\alpha,R}(u) \| u\|_{L^3}^3
& \lesssim g_{\alpha,R}(u) \| u\|_{H^{1/3}}^3
\lesssim_\alpha g_{\alpha,R}(u) \| \nabla u\|_{L^2}^{\frac{1+3\alpha}{1+\alpha}} \| u\|_{H^{-\alpha}}^{\frac{2}{1+\alpha}}\\
& \leq \|\nabla u\|_{L^2}^2 + C_\alpha\, g_{\alpha,R} (u)^{\frac{2+2\alpha}{1-\alpha}} \| u\|_{H^{-\alpha}}^{\frac{4}{1-\alpha}},
\end{align*}
where in the last passage Young's inequality is allowed under the condition $(1+3\alpha)/(1+\alpha)<2$, which holds since $\alpha<1$. Together with the properties of $g_{\alpha,R}$, this gives the estimate
\[
\frac{\d}{\d t} \|u\|_{L^2}^2 + \| \nabla u\|_{L^2}^2 \leq 2\bar{\rho}\, \|u\|_{L^2}^2 + C_\alpha\, (R+1)^{\frac{4}{1-\alpha}}
\]
and the conclusion follows from Gronwall's lemma.
\end{proof}

\begin{remark}
\label{rem:keller-segel} If $u$ solves \eqref{eq:keller-segel-cutoff-new} on
$[0,T]$ and satisfies $\sup_{t\in[0,T]} \| u_{t}\|_{H^{-\alpha}} \leq R$, then
it also solves the equation \eqref{eq:stoch-KSE} without cut-off. Due to the
freedom in choosing $R$, this shows that $u$ blows up if and only if
$\|u\|_{H^{-\alpha}}$ does. A similar reasoning applies if we consider
equation \eqref{eq:keller-segel-enhanced-new} with cut-off $g_{\alpha
,R}(\tilde{u})$ in front of the nonlinearity; in particular, if $\kappa$ is
chosen to be a function of $\rho_{0}$ as in Lemma \ref{lem:basin-KS-new} and
$R\geq\| \rho_{0}\|_{L^{2}}$, then the solutions to
\eqref{eq:keller-segel-enhanced-new} and to the PDE with cut-off coincide.
\end{remark}

We conclude this section with an analytical lemma.

\begin{lemma}
\label{lem:estim-transport-keller} Let $R>0$, $\alpha\in(0,1)$ and
$g_{\alpha,R}$ be as above; set $F(u) := g_{\alpha,R}(u) \nabla\cdot
[u\nabla^{-1} u]$. Then we have
\begin{equation}
\label{eq:estim-transport-keller}\| F(u) - F(\tilde{u})\|_{H^{-\alpha-1}}
\lesssim_{\alpha}\| u-\tilde{u}\|_{H^{-\alpha}} \big(1+\|u\|_{L^{2}}^{2} + \|
\tilde{u}\|_{H^{1}} \big) \quad\forall\, u\in L^{2}, \tilde{u}\in H^{1}.
\end{equation}

\end{lemma}

\begin{proof}
It holds $\| F(u)-F(\tilde{u})\|_{H^{-\alpha-1}} \leq I_1 + I_2$ for
\begin{align*}
I_1 = |g_{\alpha,R} (u) - g_{\alpha,R}(\tilde{u})|\, \| \nabla\cdot [u\nabla^{-1} u]\|_{H^{-\alpha-1}}, \quad
I_2 = |g_{\alpha,R}(\tilde{u})|\, \| \nabla\cdot [u\nabla^{-1} u - \tilde{u} \nabla^{-1}\tilde{u}] \|_{H^{-\alpha-1}}.
\end{align*}
The first term can be estimated by
\begin{align*}
I_1 & \leq \| g_R\|_{Lip} \| u-\tilde{u}\|_{H^{-\alpha}} \| u\nabla^{-1} u\|_{H^{-\alpha}}\\
& \lesssim_\alpha \| u-\tilde{u}\|_{H^{-\alpha}} \| u\|_{L^2} \| \nabla^{-1} u\|_{H^{1-\alpha}}\\
&\lesssim_\alpha \| u-\tilde{u}\|_{H^{-\alpha}} \| u\|_{L^2}^2,
\end{align*}
where the second passage follows from a similar proof of Lemma \ref{lem-estimate}(d). Using the property $\|g_{\alpha,R}\|_{\infty}\leq 1$ and going through computations similar to Lemma \ref{lem:estim-transport-euler}, we have
\[\aligned
I_2 &\lesssim \|u\nabla^{-1} u - \tilde{u} \nabla^{-1}\tilde{u}\|_{H^{-\alpha}} \\
&\leq \|u\nabla^{-1} (u - \tilde{u})\|_{H^{-\alpha}} +  \|(u - \tilde{u}) \nabla^{-1} \tilde{u}\|_{H^{-\alpha}} \\
&\lesssim_\alpha \|u \|_{L^2} \| u - \tilde{u} \|_{H^{-\alpha}} + \| u - \tilde{u} \|_{H^{-\alpha}} \|\nabla^{-1} \tilde{u}\|_{H^2} \\
&\leq  \| u-\tilde{u}\|_{H^{-\alpha}} (\|u\|_{L^2}+ \|\tilde{u}\|_{H^1})
\endaligned \]
and the conclusion follows.
\end{proof}

\subsection{Proof of Theorem \ref{thm-3}}

\label{subsec-KS-proof}

We now fix parameters $\varepsilon, p, L, T$ and pass to the proof of main
theorem of this section. Given the constant $C$ as in Lemma
\ref{lem:basin-KS-new}, we fix $\kappa\geq C L^{2} + 1$; we also choose
parameters $\alpha= 1-\varepsilon/2$ and $R= 2L$. Rather than looking directly
at the solution to the SPDE \eqref{eq:stoch-keller-segel-new}, we will compare
the solution $u$ to
\begin{align*}
\mathrm{d} u  &  = \big\{ \Delta u + \bar{\rho}\, u - g_{\alpha,R}(u)
\nabla\cdot[u\nabla^{-1} u] \big\}\, \mathrm{d} t + \circ\mathrm{d}
W\cdot\nabla u\\
&  = \big\{ (1+\kappa) \Delta u + \bar{\rho}\, u - g_{\alpha,R}(u) \nabla
\cdot[u\nabla^{-1} u] \big\}\, \mathrm{d} t + \mathrm{d} W\cdot\nabla u
\end{align*}
and $\tilde{u}$ to
\[
\partial_{t} \tilde{u} = (1+\kappa)\Delta u + \bar{\rho}\, u - g_{\alpha,R}(u)
\nabla\cdot[u\nabla^{-1} u]
\]
for the choice of $\alpha, R, \kappa$ as above; the noise $W$ is determined by
$(\kappa,\theta)$ with $\kappa$ as above and $\theta\in\ell^{2}$ satisfying
usual assumptions. Both equations are considered with initial data $\rho_{0}$
satisfying $\| \rho_{0}\|_{L^{2}} \leq L$ (which implies $|\bar{\rho}|\,\vee\|
u_{0}\|_{L^{2}}\leq L$ as well).

It readily follows from Remark \ref{rem:keller-segel} that the solution
$\tilde{u}$ coincides with the one to \eqref{eq:keller-segel-enhanced-new}
which satisfies \eqref{eq:basin-KS-new}. Moreover our choice of parameters and
Lemma \ref{lem:keller-segel-cutoff-new} imply that
\begin{equation}
\label{eq:estim-keller-segel-cutoff-2}\sup_{t\in[0,T]}\| u_{t}\|_{L^{2}} \leq
K_{\varepsilon, L, T}:= e^{LT} \bigg[L + \bigg(\frac{C_{1- \varepsilon/2}%
}{\bar\rho} \bigg)^{1/2} \,(2L+1)^{4/\varepsilon}\bigg].
\end{equation}
In the following, we are still going to use the parameter $\alpha$, but we ask
the reader to keep in mind that it is given by $\alpha= 1-\varepsilon/2$.

With these preparations, we are now ready to give the

\begin{proof}[Proof of Theorem \ref{thm-3}]
First observe that, if $\sup_{t\in [0,T]} \| u_t\|_{H^{-\alpha}} \leq 2L=R$, then $u$ solves the stochastic Keller-Segel equation without cutoff and so $\tau(\rho_0; \kappa,\theta) \geq T$. In other terms
\[
\P(\tau(\rho_0; \kappa,\theta) < T) \leq \P\bigg( \sup_{t\in [0,T]} \| u_t\|_{H^{-\alpha}} > 2L \bigg);
\]
furthermore, under the condition $\kappa \geq C L^2 +1$, we know that $\tilde{u}$ is a solution to the deterministic PDE without cutoff and satisfies $\sup_{t\geq 0} \| \tilde u_t\|_{H^{-\alpha}} \leq \sup_{t\geq 0} \| \tilde u_t\|_{L^2} \leq L$. Set $\xi= u-\tilde{u}$, then by triangular inequality
$$ \| u\|_{H^{-\alpha}}\leq \| \xi\|_{H^{-\alpha}}+ \|\tilde{u}\|_{H^{-\alpha}} \leq \| \xi\|_{H^{-\alpha}}+ L;$$
therefore
\[
\P(\tau(\rho_0; \kappa,\theta) < T) \leq \P\bigg( \sup_{t\in [0,T]} \| \xi_t\|_{H^{-\alpha}} > L \bigg)\leq \frac{1}{L^p} \E \bigg[ \sup_{t\in [0,T]} \| \xi_t\|_{H^{-\alpha}}^p \bigg],
\]
where we applied Markov's inequality. It only remains to estimate the right-hand side.
Passing to the mild formulation as usual and defining $F$ as in Lemma \ref{lem:estim-transport-keller}, we can write the equation for $\xi$ as
\begin{equation*}
\xi_t = \int_0^t e^{(1+\kappa)(t-s)\Delta} \big[\bar{\rho}\, \xi_s - F(u_s) +F(\tilde{u}_s) \big]\, \d s + Z_t,
\end{equation*}
where
\begin{equation*}
Z_t = \int_0^t e^{(1+\kappa)(t-s)\Delta}\, \d W_s\cdot \nabla u_s.
\end{equation*}
By Lemmas \ref{lem:heat-kernel} and \ref{lem:estim-transport-keller} we can find a constant $C_\alpha = \tilde{C}_\eps$ such that
\begin{align*}
\| \xi_t\|_{H^{-\alpha}}^2
& \leq \frac{\tilde{C}_\eps}{1+\kappa}\int_0^t \| \xi_s\|_{H^{-\alpha}}^2 \big(1+ \bar{\rho}^2 + \|u_s\|_{L^2}^4 + \|\tilde{u}_s\|_{H^1}^2 \big)\, \d s + \| Z_t\|_{H^{-\alpha}}^2\\
& \leq \tilde{C}_\eps \int_0^t \| \xi_s\|_{H^{-\alpha}}^2 \big(1+L^2+ \|u_s\|_{L^2}^4 + \|\tilde{u}_s\|_{H^1}^2 \big)\, \d s + \| Z_t\|_{H^{-\alpha}}^2.
\end{align*}
Applying Gronwall's lemma, together with the estimates \eqref{eq:basin-KS-new} and \eqref{eq:estim-keller-segel-cutoff-2}, we get
\begin{align*}
\sup_{t\in [0,T]} \| \xi_t\|_{H^{-\alpha}} \leq \exp\Big(\tilde{C}_\eps \big[T(1+L^2 + K_{\eps,L,T}^4) + L^2 \big] \Big) \sup_{t\in [0,T]} \| Z_t\|_{H^{-\alpha}} =: K'_{\eps,L,T} \sup_{t\in [0,T]} \| Z_t\|_{H^{-\alpha}}.
\end{align*}
Finally, we can apply Corollary \ref{cor-stoch-convol} to $Z$ for the choice $\beta= \alpha$, $\tilde{\eps}= \eps/2$ (so that $\beta-\tilde{\eps} = 1-\eps$) and $\delta=1+\kappa\sim \kappa$ (recall that $\kappa\geq 1$) to obtain
\begin{align*}
\E \bigg[ \sup_{t\in [0,T]} \| \xi_t\|_{H^{-\alpha}}^p \bigg] \lesssim_{\eps,p, T}  (K'_{\eps,L,T})^p\, \kappa^{\eps p/4} \| \theta\|_{\ell^\infty}^{p(1-\eps)}.
\end{align*}
Combining everything together we arrive at
\begin{align*}
\P(\tau(\rho_0;\kappa,\theta))<T) \leq C_2\, \kappa^{\eps p/4} \| \theta\|_{\ell^\infty}^{p(1-\eps)}.
\end{align*}
where the constant $C_2=C_2(\eps,p,L,T)$ can be calculated explicitly in terms of the ones which appeared previously. Taking $C_1$ as the constant from Lemma \ref{lem:basin-KS-new} gives the conclusion.
\end{proof}

\section{Proofs in the linear case}\label{sec:transport}

This section consists of two parts: in Section \ref{subs-inviscid-transport} we prove Theorem \ref{thm-transport} and some related results, while in Section \ref{subs-dissip-enhanc-proof} we prove the exponential decay of $L^2$-norm at infinity for solutions to the transport-diffusion equation \eqref{heat transport eq}.

\subsection{Proofs of Theorem \ref{thm-transport} and related results}\label{subs-inviscid-transport}

We first provide the

\begin{proof}[Proof of Theorem \ref{thm-transport}]
Denote by $P_{t}=e^{t\kappa\Delta}$ the heat semigroup on $\T^d$; using the mild formulation of equations \eqref{stoch-transp-Ito} and \eqref{heat-eq}, we have%
\begin{align*}
\left\langle f_{t},\phi\right\rangle -\left\langle \overline{f}_{t}%
,\phi\right\rangle  &  = -\sqrt{C_d \kappa}\sum_{k,i}\theta
_{k}\,\int_{0}^{t}\left\langle P_{t-s}\left(  \sigma_{k,i}\cdot\nabla
f_{s}\right)  ,\phi\right\rangle \d W_{s}^{k,i}\\
&  = \sqrt{C_d \kappa}\sum_{k,i}\theta_{k}\,\int_{0}%
^{t}\left\langle f_{s},\sigma_{k,i}\cdot\nabla P_{t-s}\phi\right\rangle
\d W_{s}^{k,i}
\end{align*}
and thus %
\begin{align*}
\mathbb{E}\left[  \left\vert \left\langle f_{t},\phi\right\rangle
-\left\langle \overline{f}_{t},\phi\right\rangle \right\vert ^{2}\right]
&  = C_d \kappa\sum_{k,i}\theta_{k}^{2}\, \mathbb{E}\int_{0}%
^{t} | \langle f_{s},\sigma_{k,i}\cdot\nabla P_{t-s}\phi \rangle |^{2}\, \d s.
\end{align*}
Denote by $g_{s,t}(x)$ the function $f_{s}(x) \left(  \nabla P_{t-s}\phi\right)
(x) $; since $C_d = d/(d-1)\leq 2$ and $\{\sigma_{k,i}\}_{k,i}$ is an orthonormal system in $L^2(\T^d;\R^d)$, we have
\begin{align*}
\mathbb{E}\left[  \left\vert \left\langle f_{t},\phi\right\rangle
-\left\langle \overline{f}_{t},\phi\right\rangle \right\vert ^{2}\right]
& \leq 2 \kappa\sum_{k,i}\theta_{k}^{2}\, \mathbb{E}\int_{0}%
^{t}\left\vert\langle \sigma_{k,i}, g_{s,t}\rangle  \right\vert
^{2} \,\d s\\
& \leq 2\kappa\, \| \theta\|_{\ell^{\infty}}^2\,  \E \bigg[ \sum_{k,i}\int
_{0}^{t}\left\vert\langle \sigma_{k,i}, g_{s,t}\rangle  \right\vert^{2}\,\d s \bigg]%
 \\
& \leq2\kappa\,\Vert\theta\Vert_{\ell^{\infty}}^{2}\,\mathbb{E}\int_{0}%
^{t}\left\Vert f_{s}\left(  \nabla P_{t-s}\phi\right)  \right\Vert _{L^{2}%
}^{2} \, \d s.
\end{align*}
Now we use the $\P$-a.s. inequality $\Vert f_{t}\Vert_{L^{\infty}}\leq\Vert
f_{0}\Vert_{L^{\infty}}$ from \eqref{solu-bounds transp} to get
\[
\mathbb{E}\left[  \left\vert \left\langle f_{t},\phi\right\rangle
-\left\langle \overline{f}_{t},\phi\right\rangle \right\vert ^{2}\right]
\leq2\kappa\Vert\theta\Vert_{\ell^{\infty}}^{2}\Vert f_{0}\Vert_{L^{\infty}%
}^2 \int_{0}^{t}\left\Vert \nabla P_{t-s}\phi\right\Vert _{L^{2}}^{2} \,\d s.
\]
Finally,%
\begin{align*}
2\kappa\int_{0}^{t}\left\Vert \nabla P_{t-s}\phi\right\Vert _{L^{2}}^{2} \,\d s &
=-\int_{0}^{t}2\left\langle \kappa\Delta P_{s}\phi,P_{s}\phi\right\rangle \d s\
=\int_{0}^{t}\frac{\d}{\d s}\left\Vert P_{s}\phi\right\Vert _{L^{2}}^{2}%
\d s \leq\left\Vert \phi\right\Vert _{L^{2}}^{2}.
\end{align*}
This completes the proof of estimate \eqref{eq:thm-transport-eq1}; estimate \eqref{smeared} follows by taking, for every $x_{0}\in\mathbb{T}^{2}$,
$\phi_{x_{0}}\left(  x\right)  :=\chi\left(  x_{0}-x\right)  $; thus we get
\[
\mathbb{E}\left[  \left\vert \left(  \chi\ast f_{t}\right)  \left(  x_{0}\right)
-\left(  \chi\ast\overline{f}_{t}\right)  \left(  x_{0}\right)  \right\vert
^{2}\right]  \leq\Vert\theta\Vert_{\ell^{\infty}}^{2}\Vert f_{0}%
\Vert_{L^{\infty}}^{2}\Vert \chi\left(  x_{0}-\cdot\right)  \Vert_{L^{2}}^{2}.
\]
Integrating in $x_{0}$ we deduce the second inequality of the theorem.
\end{proof}

Compared to other results in this paper, Theorem \ref{thm-transport} has the nice feature that it does not produce any constants depending on $t$ or $\kappa$; this comes at the price of imposing higher regularity on the initial data $f_0\in L^\infty$ and obtaining a probabilistic estimate which depends on the given $f_0,\phi$ in consideration.

For this reason, we will now complement Theorem \ref{thm-transport} with another result quantifying the distance of the random solution operator associated to \eqref{stoch-transp-Ito} from the heat kernel operator $P_t=e^{t \kappa \Delta}$, in some weak norm.

Before giving the statement, we need some preparations. In the remainder of the section for simplicity we will assume $\theta$ to enjoy suitable summability (as before, $\sum_k |k|^2\theta_k^2<\infty$ would suffice), so that we can construct the incompressible stochastic flow $\{X_t\}_{t\geq 0}$ associated to $W$ and represent any solution $f$ to \eqref{stoch-transp-Ito} by $f_t(x) = f_0(X^{-1}_t(x))$.
We can then define the random solution operator $S_t \varphi := \varphi\circ X^{-1}_t$, which by incompressibility of $X_t$ is a family of isomorphisms of $L^p(\T^d)$ for any $p \in [1,\infty]$.

Next, let us recall that, given any two Hilbert spaces $E_1, E_2$, a linear operator $A:E_1\to E_2$ is Hilbert--Schmidt, $A\in \L^2(E_1,E_2)$, if
\begin{equation*}
\| A\|_{\L^2(E_1;E_2)}^2=\sum_n \| A \varphi_n\|_{E_2}^2<\infty
\end{equation*}
for some (equivalently any) $\{\varphi_n\}_n$ CONS of $E_1$; in this case $\| A\|_{E_1\to E_2} \leq \| A\|_{\L^2(E_1;E_2)}$.

\begin{proposition}\label{prop:fin-time-mixing}
For any $s>d/2$, $T>0$, $\alpha>0$, $p\in [2,\infty)$ and $\eps\in (0,\alpha)$ it holds
\begin{equation}\label{eq:inviscid-mixing}
\E\bigg[\sup_{t\in [0,T]} \| S_t - P_t \|_{\L^2(H^s,H^{-\alpha})}^p\bigg]^{1/p} \lesssim_{s,\eps,T,p} \kappa^{\eps/2}\, \| \theta\|_{\ell^\infty}^{2(\alpha-\eps)/d}.
\end{equation}
\end{proposition}

\begin{proof}
Given $f$ solution to \eqref{stoch-transp-Ito} with initial data $f_0$, passing to mild formulation we have
\[
(S_t-P_t) f_0
= \int_0^t e^{(t-s)\kappa\Delta} \nabla f_s\cdot \d W_s =: Z^{f_0}_s
\]
which is a stochastic convolution as the ones treated in Section \ref{subsec-stoch-convol}. Moreover, given any CONS $\{\varphi_n\}_n$ of $H^s$, denoting by $Z^{\varphi_n}$ the associated processes, it holds
\begin{align*}
\sup_{t\in [0,T]} \| S_t - P_t \|_{\L^2(H^s,H^{-\alpha})}
\leq \bigg[ \sum_n \sup_{t\in [0,T]} \|(S_t - P_t)\,\varphi_n \|_{H^{-\alpha}}^2\bigg]^{1/2}
= \bigg[ \sum_n \sup_{t\in [0,T]} \| Z^{\varphi_n}_t\|_{H^{-\alpha}}^2\bigg]^{1/2}.
\end{align*}
Now choose as a CONS of $H^s$ the family $g_k = (1+|k|^2)^{-s/2} e_k$ for $k\in \Z^d_0$, then applying the above estimates, together with Minkowski's inequality and Corollary \ref{cor-stoch-convol}, we obtain
\begin{align*}
\E\bigg[ \sup_{t\in [0,T]} \| S_t - P_t \|_{\L^2(H^s,H^{-\alpha})}^p\bigg]^{2/p}
& \lesssim \sum_k \E \bigg[\sup_{t\in [0,T]} \| Z^{g_k}_t\|_{H^{-\alpha}}^p  \bigg]^{2/p}\\
& \lesssim_{\eps,\kappa,p,T} \kappa^{\eps} \| \theta\|_{\ell^\infty}^{4(\alpha-\eps)/d} \sum_k \| g_k\|_{L^2}^2\\
& \lesssim_{\eps,\kappa,p,T} \kappa^{\eps} \| \theta\|_{\ell^\infty}^{4(\alpha-\eps)/d} \sum_k(1+|k|^2)^{-s}
\end{align*}
which gives the conclusion.
\end{proof}

Compared to Theorem \ref{thm-transport}, estimate \eqref{eq:inviscid-mixing} depends on several parameters and requires the use of the strong norm $H^s$; but it gives a bound on the random operator $S_t$ and thus on
\begin{align*}
\langle f_t -\bar{f}_t,\phi\rangle
= \langle (S_t-P_t) f_0,\phi\rangle
= \int_{\T^d} f_0(x) \big(\phi (X^x_t) - \E[\phi(X^x_t)] \big)\, \d x
\end{align*}
uniformly over all possible $f_0\in H^s$, $\phi\in H^\alpha$ at once, thus revealing more information on the behaviour of the stochastic flow $X^x_t$ as well.

\begin{remark}
The property $\|\cdot\|_{H^s\to H^{-\alpha}} \lesssim \| \cdot\|_{\L^2(H^s;H^\alpha)}$, combined with estimate \eqref{eq:inviscid-mixing} and Markov's inequality, yields
\begin{equation*}
\P\bigg(\sup_{t\in [0,T]} \| S_t-P_t\|_{H^s\to H^{-\alpha}}>\delta\bigg) \lesssim \delta^{-p} \kappa^{\eps p/2} \| \theta\|_{\ell^\infty}^{2(\alpha-\eps)p/d}
\end{equation*}
for all $\delta >0$; in particular, for suitable chosen $(\kappa,\theta)$ the quantity $\sup_{t\in [0,T]} \| S_t-P_t\|_{H^s\to H^{-\alpha}}$ is very small with high probability. Moreover this can be attained while choosing $\kappa$ arbitrarily large, so that $\|P_t\|_{H^s\to H^{-\alpha}}$ becomes arbitrarily small as well (for $t\geq t_0> 0$), implying a probabilistic bound for $\|S_t\|_{H^s\to H^{-\alpha}}$ as well. Finally, interpolating the estimate on $\| S_t-P_t\|_{H^s\to H^{-\alpha}}$ with the $\P$-a.s. one $\|S_t-P_t\|_{L^2\to L^2}\leq 2$, we can deduce similar bounds for $\| S_t-P_t\|_{H^{s'}\to H^{-\alpha}}$ with $s'\in (0,s)$, thus removing the restriction $s>d/2$.
\end{remark}

\subsection{Proof of Theorem \ref{thm-transp-diffus}} \label{subs-dissip-enhanc-proof}

We first briefly recall the setting. We consider the linear transport-diffusion equation
\begin{equation*}
\d f + \circ \d W\cdot \nabla f = \nu \Delta f\, \d t
\end{equation*}
with $\nu>0$; it admits the It\^o formulation
  $$\d f + \d W\cdot \nabla f = (\kappa+\nu) \Delta f\, \d t. $$
For any $f_0\in L^2(\T^d)$, it is well known that the equation has a unique solution $f$ satisfying: $\P$-a.s., $f\in C([0,+\infty);L^2)\cap L^2(0,+\infty;{H}^1)$.

Below we assume $f_0$ has zero mean, a property preserved by the solution $\{f_t\}_{t\geq 0}$; set $P_t= e^{t(\kappa+\nu)\Delta}$. For any $0\leq s<t$, we have the mild formulation
  \begin{equation}\label{mild-form-1}
  f_t = P_{t-s} f_s + Z_{s,t},
  \end{equation}
where
  \begin{equation}\label{stoch-convol}
  Z_{s,t}:= - \sqrt{2\kappa} \sum_{k,i} \theta_k \int_s^t P_{t-r} (\sigma_{k,i}\cdot\nabla f_r)\,\d W^{k,i}_r;
  \end{equation}
and
  \begin{equation}\label{energy-equality}
  \P \mbox{-a.s.}, \quad \| f_t\|_{L^2}^2 + 2 \nu \int_s^t \| \nabla f_r\|_{L^2}^2\, \d r = \| f_s\|_{L^2}^2.
  \end{equation}
This implies that $t\to \| f_t\|_{L^2}$ is almost surely decreasing.

In order to get estimates on $\| f_t\|_{L^2}$, the key is to estimate $\| Z_{s,t} \|_{L^2}$; due to the linear structure, here we directly estimate $\E \|Z_{s,t} \|_{L^2}^2 $ without applying Gr\"onwall's lemma, contrary to the nonlinear case.

\begin{lemma}\label{lem-estim-interv}
There exists $\delta \in (0,1)$ such that, for any $n\geq 0$,
  $$\E \| f_{n+1} \|_{L^2}^2 \leq \delta\, \E \| f_n\|_{L^2}^2,$$
where, for some $0<\alpha <1 \leq \frac d2 <\beta$,
  $$\delta \lesssim_{\alpha, \beta} \kappa^{-1} + \kappa^{\frac{4\beta- \alpha(2\beta+d)}{4(\alpha+\beta)}} \nu^{-\frac{\beta}{\alpha+\beta}} \|\theta \|_{\ell^\infty}^{\frac{2\alpha}{\alpha+\beta}}. $$
In particular, $\delta$ can be as small as we want by first taking $\kappa$ big and then choosing $\theta\in \ell^2(\Z_0^d)$ with $\|\theta \|_{\ell^\infty}$ small enough.
\end{lemma}

\begin{proof}
Since $\| f_t\|_{L^2}$ is decreasing in $t$, we have, by \eqref{mild-form-1},
  $$\| f_{n+1} \|_{L^2}^2 \leq \int_{n}^{n+1} \|f_t \|_{L^2}^2 \,\d t \leq 2 \int_{n}^{n+1} \|P_{t-n} f_n \|_{L^2}^2 \,\d t + 2 \int_{n}^{n+1} \|Z_{n,t} \|_{L^2}^2 \,\d t. $$
First,
  \begin{equation}\label{lem-estim-interv.0}
  \int_{n}^{n+1} \|P_{t-n} f_n \|_{L^2}^2 \,\d t \leq \int_{n}^{n+1} e^{-8\pi^2 (\kappa+\nu)(t-n)} \|f_n \|_{L^2}^2 \,\d t \lesssim \frac{\|f_n \|_{L^2}^2}{\kappa+\nu}.
  \end{equation}
Next, we turn to estimate the second term for which we use an interpolation argument. Fix an $\alpha\in (0,1)$, we have, by \eqref{stoch-convol},
  $$\aligned
  \int_{n}^{n+1} \E \|Z_{n,t} \|_{H^\alpha}^2 \,\d t &= 2\kappa \int_{n}^{n+1} \E \bigg\|\sum_{k,i} \theta_k \int_n^t P_{t-r} (\sigma_{k,i}\cdot\nabla f_r)\,\d W^{k,i}_r \bigg\|_{H^\alpha}^2 \,\d t \\
  &\lesssim \kappa \int_{n}^{n+1} \E \bigg[\sum_{k,i} \theta_k^2 \int_n^t \| P_{t-r} (\sigma_{k,i}\cdot\nabla f_r) \|_{H^\alpha}^2 \,\d r \bigg] \,\d t \\
  &\lesssim \frac{\kappa}{(\kappa+\nu)^\alpha} \int_{n}^{n+1} \E \bigg[\sum_{k,i} \theta_k^2 \int_n^t \frac1{(t-r)^\alpha} \| \sigma_{k,i} \cdot\nabla f_r \|_{L^2}^2 \,\d r \bigg] \,\d t,
  \endaligned $$
where the last step follows from Lemma \ref{lem:heat-kernel-estim}(i). Using the fact $\| \sigma_{k,i} \cdot\nabla f_r \|_{L^2} \leq \| \nabla f_r \|_{L^2}$, we obtain
  $$\aligned
  \int_{n}^{n+1} \E \|Z_{n,t} \|_{H^\alpha}^2 \,\d t &\lesssim \frac{\kappa}{(\kappa+\nu)^\alpha} \|\theta \|_{\ell^2}^2 \int_{n}^{n+1} \E \int_n^t \frac1{(t-r)^\alpha} \| \nabla f_r \|_{L^2}^2 \,\d r \d t \\
  &\leq \kappa^{1-\alpha} \|\theta \|_{\ell^2}^2\, \E \int_{n}^{n+1} \| \nabla f_r \|_{L^2}^2\, \d r \int_r^{n+1} \frac1{(t-r)^\alpha} \, \d t \\
  &\leq \frac{\kappa^{1-\alpha}}{1-\alpha} \E \int_{n}^{n+1} \| \nabla f_r \|_{L^2}^2\, \d r,
  \endaligned $$
where in the last step we have used $\|\theta \|_{\ell^2}=1$. Now by \eqref{energy-equality} we arrive at
  \begin{equation}\label{lem-estim-interv.1}
  \int_{n}^{n+1} \E \|Z_{n,t} \|_{H^\alpha}^2 \,\d t \lesssim_\alpha \kappa^{1-\alpha} \nu^{-1} \E \|f_n\|_{L^2}^2.
  \end{equation}

Next, for $\beta> d/2$, we define $\eps=(2\beta -d)/4>0$; similarly to the proof of \eqref{eq:estim-Z-ellinfty}, we have
  $$\aligned
  \E \|Z_{n,t} \|_{H^{-\beta}}^2 &\leq 2\kappa\, \E\bigg[\sum_{k,i} \theta_k^2 \int_n^t \| P_{t-r}(\sigma_{k,i}\cdot \nabla f_r)\|_{H^{-d/2-2\eps}}^2 \,\d r \bigg] \\
  &\lesssim \frac{\kappa}{(\kappa +\nu)^{1-\eps}} \E\bigg[\sum_{k,i} \theta_k^2 \int_n^t \frac{\| \sigma_{k,i} \cdot \nabla f_r\|_{H^{-1-d/2-\eps}}^2}{(t-r)^{1-\eps}} \,\d r \bigg] \\
  &\lesssim_\eps \kappa^{\eps} \|\theta \|_{\ell^\infty}^2\, \E \|f_n\|_{L^2}^2 .
  \endaligned $$
Thus, noting that $\eps= (2\beta -d)/4$,
  \begin{equation}\label{lem-estim-interv.2}
  \int_{n}^{n+1} \E \|Z_{n,t} \|_{H^{-\beta}}^2 \,\d t \lesssim_\beta \kappa^{(2\beta -d)/4} \|\theta \|_{\ell^\infty}^2\, \E \|f_n\|_{L^2}^2.
  \end{equation}

Finally, for $0<\alpha <1 \leq \frac d2 <\beta$, by interpolation
  $$\|\phi \|_{L^2} \lesssim \|\phi \|_{H^\alpha}^{\beta/(\alpha+\beta)} \|\phi \|_{H^{-\beta}}^{\alpha/(\alpha+\beta)}, \quad \forall\, \phi \in H^\alpha, $$
we have
  $$\aligned
  \int_{n}^{n+1} \E \|Z_{n,t} \|_{L^2}^2 \,\d t &\lesssim \int_{n}^{n+1} \E\Big( \|Z_{n,t} \|_{H^\alpha}^{2\beta/(\alpha+\beta)} \|Z_{n,t} \|_{H^{-\beta}}^{2\alpha/(\alpha+\beta)} \Big) \,\d t \\
  &\leq \bigg(\int_{n}^{n+1} \E \|Z_{n,t} \|_{H^\alpha}^{2} \,\d t \bigg)^{\frac{\beta}{\alpha+\beta}} \bigg(\int_{n}^{n+1} \E \|Z_{n,t} \|_{H^{-\beta}}^{2} \,\d t \bigg)^{\frac{\alpha}{\alpha+\beta}}
  \endaligned $$
by H\"older's inequality. Inserting \eqref{lem-estim-interv.1} and \eqref{lem-estim-interv.2} into this estimate leads to
  $$\int_{n}^{n+1} \E \|Z_{n,t} \|_{L^2}^2 \,\d t \leq \kappa^{\frac{4\beta- \alpha(2\beta+d)}{4(\alpha+\beta)}} \nu^{-\frac{\beta}{\alpha+\beta}} \|\theta \|_{\ell^\infty}^{\frac{2\alpha}{\alpha+\beta}} \E\|f_n\|_{L^2}^2. $$
Combining this estimate with \eqref{lem-estim-interv.0}, we complete the proof.
\end{proof}

We can now provide

\begin{proof}[Proof of Theorem \ref{thm-transp-diffus}]
Lemma \ref{lem-estim-interv} implies that there exists a small $\delta \in (0,1)$ such that for any $n\geq 1$,
  $$\E \|f_n \|_{L^2}^2 \leq \delta\, \E \|f_{n-1} \|_{L^2}^2 \leq \cdots \leq \delta^n \|f_0 \|_{L^2}^2. $$
Recall that $t\to \|f_t \|_{L^2}$ is $\P$-a.s. decreasing, we have
  $$\E \bigg(\sup_{t\in [n,n+1]} \|f_t \|_{L^2}^2 \bigg) \leq \E \|f_n \|_{L^2}^2 \leq \delta^n \|f_0 \|_{L^2}^2 = e^{-2\lambda' n} \|f_0 \|_{L^2}^2, $$
where $\lambda'= -\frac12 \log \delta >0$. By Lemma \ref{lem-estim-interv}, we can choose a suitable pair $(\kappa,\theta)$ such that $\lambda'> \lambda(1+p/2)$, where $\lambda>0$ and $p\geq 1$ are parameters in the statement of Theorem \ref{thm-transp-diffus}.

Now for any $n\geq 1$, we define
  $$A_n = \bigg\{\omega\in \Omega: \sup_{t\in [n,n+1]} \|f_t(\omega) \|_{L^2} > e^{-\lambda n} \|f_0 \|_{L^2} \bigg\}.$$
Then by Chebyshev's inequality,
  $$\sum_n \P(A_n) \leq \sum_n \frac{e^{2\lambda n}}{\|f_0 \|_{L^2}^2} \E\bigg(\sup_{t\in [n,n+1]} \|f_t \|_{L^2}^2 \bigg) \leq \sum_n e^{2(\lambda -\lambda')n} <+\infty, $$
therefore, by Borel-Cantelli lemma, for $\P$-a.e. $\omega\in \Omega$, there exists a big $N(\omega) \geq 1$ such that
  $$\sup_{t\in [n,n+1]} \|f_t(\omega) \|_{L^2} \leq e^{-\lambda n} \|f_0 \|_{L^2} \quad \forall\, n> N(\omega).$$
For $0\leq n\leq N(\omega)$, we have
  $$\sup_{t\in [n,n+1]} \|f_t(\omega) \|_{L^2} \leq \|f_n(\omega) \|_{L^2} = e^{\lambda n} e^{-\lambda n} \|f_n(\omega) \|_{L^2} \leq e^{\lambda N(\omega)} e^{-\lambda n} \|f_0 \|_{L^2}. $$
Thus, if we take $C(\omega)= e^{\lambda (1+N(\omega))} $, then it is easy to show that, $\P$-a.s. for all $t\geq 0$, $\|f_t(\omega) \|_{L^2} \leq C(\omega) e^{-\lambda t} \|f_0 \|_{L^2}$.

It remains to estimate the $p$-th moment of the random variable $C(\omega)$; to this end, we need to estimate the tail probability $\P(\{N(\omega) \geq k\})$. Note that $N(\omega)$ may be defined as the largest integer $n$ such that $\sup_{t\in [n,n+1]} \|f_t(\omega) \|_{L^2} > e^{-\lambda n} \|f_0 \|_{L^2}$; hence
  $$\{\omega\in \Omega: N(\omega) \geq k\} = \bigcup_{n=k}^\infty A_n. $$
Then, we have
  $$\P(\{N(\omega) \geq k\}) \leq \sum_{n=k}^\infty \P(A_n) \leq \sum_{n=k}^\infty e^{2(\lambda -\lambda')n} = \frac{e^{2(\lambda -\lambda')k}}{1- e^{2(\lambda -\lambda')}}. $$
As a result,
  $$\E e^{\lambda p N(\omega)} = \sum_{n=0}^\infty e^{\lambda p k} \P(\{N(\omega) = k\}) \leq \frac1{1- e^{2(\lambda -\lambda')}} \sum_{n=0}^\infty e^{\lambda p k} e^{2(\lambda -\lambda')k} < +\infty ,$$
where the last step is due to the choice of $\lambda'$. Therefore $C(\omega)$ has finite $p$-th moment.
\end{proof}

\bigskip

\noindent\textbf{Acknowledgements.}
The second author is funded by the DFG under Germany's Excellence Strategy - GZ 2047/1, project-id 390685813.
The last named author is grateful to the
financial supports of the National Key R\&D Program of China (No.
2020YFA0712700), the National Natural Science Foundation of China (Nos.
11688101, 11931004, 12090014) and the Youth Innovation Promotion Association, CAS (2017003).

\end{document}